\documentclass[12pt]{article}
\usepackage[margin=1.2in]{geometry} 
\usepackage{amsmath, amsfonts}
\usepackage{tcolorbox}
\usepackage{amssymb}
\usepackage{amsthm}
\usepackage{lastpage}
\usepackage{fancyhdr}
\usepackage{subcaption}
\usepackage{accents}
\usepackage{enumitem}
\usepackage{tikz-cd}
\usepackage{cases}
\usepackage{hyperref}
\usepackage[british]{babel}
\newtheorem{thm}{Theorem}

\newtheorem{conj}[thm]{Conjecture}
\newtheorem*{fact}{Fact}
\newtheorem{prop}[thm]{Proposition}
\newtheorem{cor}{Corollary}[thm]

\theoremstyle{definition}

\newtheorem*{dfn*}{Definition}
\newtheorem*{notation*}{Notation}

\newtheorem*{rmk*}{Remark}

\newcommand{\Z}{\mathbb{Z}}

\newcommand{\Q}{\mathbb{Q}}
\newcommand{\C}{\mathbb{C}}

\newcommand{\CP}{\mathbb{C}P}

\DeclareMathOperator{\Id}{Id}
\DeclareMathOperator{\rk}{rk}
\DeclareMathOperator{\lk}{lk}

\DeclareMathOperator{\Torsion}{Torsion}

\title{Surgeries on Iterated Torus Knots Bounding Rational Homology 4-Balls}
\date{May 2023}
\author{Lisa Lokteva\\ University of Glasgow \\ \href{mailto:e.lokteva.1@research.gla.ac.uk}{e.lokteva.1@research.gla.ac.uk}}

\begin{document}
\maketitle
\begin{abstract}
    We show that all large enough positive integral surgeries on algebraic knots bound a 4-manifold with a negative definite plumbing tree, which we describe explicitly. Then we apply the lattice embedding obstruction coming from Donaldson's Theorem to classify the ones of the form $S^3_n(T(p_1,k_1p_1+1; p_2, k_2p_2\pm 1))$ that also bound rational homology 4-balls.
\end{abstract}

\tableofcontents

\cfoot{\thepage\ of \pageref{LastPage}}

\section{Introduction}

\begin{dfn*}
    For a manifold $M$, we say that a manifold $N$ is a \textbf{rational homology $M$} if $M$ and $N$ are of the same dimension and $H_*(M; \Q) \cong H_*(N; \Q)$.
\end{dfn*}
One major problem in low-dimensional topology is to determine which rational homology 3-spheres bound rational homology 4-balls. It is attributed to Casson and appears as Problem 4.5 on Kirby's list of important problems in the discipline \cite{kirbylist}. While rational homology 3-spheres abound, very few of them tend to bound rational homology balls. This can be illustrated by the fact that while the $n$-surgery on a knot $K \subset S^3$, denoted $S^3_n(K)$, is a rational homology 3-sphere for all $n\neq 0$ and knots $K$, Aceto and Golla showed in \cite[Theorem 1.2]{acetogolla} that in fact, for each $K$, there are at most four possible integer values of $n$ such that $S^3_n(K)$ bounds a rational homology ball.

The first study of rational homology 3-spheres bounding rational homology 4-balls was published in 1981, when Casson and Harer found several families of homology lens spaces bounding rational homology 4-balls and homology 3-spheres bounding contractible manifolds \cite{casson-harer}. In 2007, Lisca classified all the lens spaces and connected sums of lens spaces bounding rational homology 4-balls \cite{liscasingle07, liscamultiple07}, popularising the technique of obstructing bounding rational homology 4-balls with lattice embeddings. Many people have since then used lattice embeddings on various classes of 3-manifolds to classify the ones that admit fillings with certain homological constraints. Examples include Lecuona's study of double branched covers of $S^3$ branched over some families of Montesinos knots \cite{lecuonamontesinos}, Aceto's study of rational homology $S^1\times S^2$'s bounding rational homology $S^1\times D^3$ \cite{aceto20}, and Simone's classifying torus bundles on the circle bounding rational homology $S^1\times D^3$ \cite{simone2020classification}, which he used to construct rational homology 3-spheres bounding rational homology 4-balls in \cite{simone2021using}. Recently, Aceto, Golla, Larson and Lecuona managed to answer the rationally acyclic filling question for positive integral surgeries on positive torus knots, a classification with a whopping 18 cases \cite{acetogolla, GALL}. 

The idea of this fruitful technique called lattice embeddings is to represent the rational homology 3-sphere as the boundary of a negative definite 4-manifold and to use the following corollary of Donaldson's theorem \cite[Theorem 1]{donaldsonthm}:

\begin{prop}
\label{donaldson}
Let $Y$ be a rational homology 3-sphere and $Y=\partial X$ for $X$ a connected smooth oriented negative definite 4-manifold. If $Y=\partial W$ for a smooth rational homology 4-ball $W$, then there exists a lattice embedding $(H_2(X)/\Torsion, Q_{X}) \hookrightarrow (\Z^{\rk H_2(X)}, -\Id)$.
\end{prop}

Here $Q_X$ is the intersection form of $X$, and lattice embeddings are defined in Section \ref{latticeembeddings}. Proposition \ref{donaldson} also has a positive version, where $X$ is positive definite and the embedding goes into $(\Z^{\rk H_2(X)}, \Id)$.

The author is trying to build on the works of Lisca, Lecuona, Aceto, Golla and Larson and classify the positive surgeries on iterated torus knots bounding rational homology balls. An iterated torus knot is a knot obtained from the unknot through repeated cabling operations.

\begin{dfn*}
    Let $K \subset S^3$ be an oriented smooth knot. The boundary $\partial (\nu K)$ of a tubular neighbourhood $\nu K$ of $K$ is an embedded torus in $S^3$. The \textbf{meridian} $M$ and the \textbf{longitude} $L$ are oriented simple closed curves inside $\partial (\nu K)$, determined up to isotopy by the following homology and linking relations: \begin{itemize}
        \item $[M]=0$ and $[L]=[K]$ in $H_1(\nu K)$, and
        \item $\lk(M, K)=1$ and $\lk(L,K)=0$.
    \end{itemize}
    
    Let $p,q$ be relatively prime integers. We denote by $C_{p,q}(K) \subset S^3$, the unique (up to isotopy) simple closed curve in $\partial (\nu K)$ with homology class $p[L]+q[M] \in H_1(\partial (\nu K))$. The curve $\C_{p,q}(K)$ is called the $(p,q)$-\textbf{cable} on $K$.
\end{dfn*}

\begin{dfn*}
    The \textbf{iterated torus knot} with $k$ iterations $T(p_1,\alpha_1;p_2,\alpha_2;\cdots;p_k,\alpha_k)$ is the knot \[T(p_1,\alpha_1;p_2,\alpha_2;\cdots;p_k,\alpha_k)=C_{p_k,\alpha_k}C_{p_{k-1}, \alpha_{k-1}}\cdots C_{p_1,\alpha_1}(O),\] $O$ being the unknot.
\end{dfn*}

Iterated torus knots are interesting to consider because many of them, just like positive torus knots, arise as links of cuspidal singularities of complex plane curves. We will call the iterated torus knots that do arise as singularity links of cuspidal curves \textbf{algebraic}. Resolving the singularity using blow-ups allows us to obtain a plumbing description of a 4-manifold with low $b_2^+$ (the number of positive eigenvalues of $Q_X$) bounding the surgery on the knot. Unfortunately, the author lacks the luxury of being able to push down $b_2^+$ to 0 as easily as for torus knots, which slightly restricts the $n$ for which we can answer the question whether or not $S^3_n(K)$ bounds a rational homology ball, excluding finitely many cases for each $K$ from our study. When we do have a negative definite filling of our 3-manifold, we need to investigate the existence of a lattice embedding prescribed by Proposition \ref{donaldson}. This can be a very difficult combinatorial problem. For example, the classification of positive integral surgeries on positive torus knots $T(p,q)$ bounding rational homology balls contains a lattice embedding analysis well over 40 pages long (\cite[Section 6]{GALL}), and this does not include the case when $q \equiv \pm 1 \pmod{p}$, which was studied in the earlier paper \cite{acetogolla}. Also, minimal changes of the intersection form can render former techniques for studying the lattice embedding useless. In this first paper on integral surgeries on iterated torus knots bounding rational homology balls, we restrict ourselves to algebraic iterated torus knots of the form $T(p_1,k_1p_1+1; p_2, k_2p_2\pm 1)$. We prove the following theorem:


\begin{thm}
\label{mainthm}
Let $\alpha_1 \equiv 1 \pmod{p_1}$, $\alpha_2 \equiv \pm 1 \pmod{p_2}$, $\alpha_2/p_2 > p_1\alpha_1$ and $n \geq 2+p_2\alpha_2$. Then the rational homology 3-sphere $S^3_n(T(p_1,\alpha_1; p_2, \alpha_2))$ bounds a rational homology 4-ball if and only if the tuple $(p_1, \alpha_1; p_2, \alpha_2; n)$ is one of the following:
\begin{enumerate}
    \item $(p_1,p_1+1; p_2, p_2(p_1+1)^2-1; p_2^2(p_1+1)^2)$ or
    \item $(2,7; p_2, 16p_2-1; 16p_2^2)$.
\end{enumerate}
\end{thm}

\begin{rmk*}
The condition $\alpha_2/p_2 > p_1\alpha_1$ is equivalent to the algebraicity of the knot, and $n \geq 2+p_2\alpha_2$ is needed in order for $S^3_n(T(p_1,\alpha_1; p_2, \alpha_2))$ to bound an H-shaped negative definite plumbing of disc bundles over spheres. The conditions $\alpha_1 \equiv 1 \pmod{p_1}$ and $\alpha_2 \equiv \pm 1 \pmod{p_2}$ are, analogously to the conditions of \cite{acetogolla}, there to simplify the lattice embedding analysis.
\end{rmk*}

It is interesting to compare this result to other work on surgeries on iterated torus knots bounding rational homology 4-balls. We have already seen examples of ones that do. One result in that vein is Theorem 1.3 of Aceto, Golla, Larson and Lecuona in \cite{GALL}, which given a surgery on a knot $K$ bounding a rational homology ball gives us surgeries on infinitely many of its cables bounding rational homology balls. Another work is of an algebro-geometric flavour. Bodnár classified in \cite{bodnar} all rational unicuspidal complex curves $C$ inside $\CP^2$ with two Newton pairs. This is relevant to us because the complement of a tubular neighbourhood of $C$, $\CP^2-\nu C$, is a rational homology 4-ball, and $\partial(\CP^2-\nu C)=\partial(\nu C)=S^3_{d^2}(K)$ for $K$ an iterated torus knot of two iterations and $d$ the degree of the curve. However, Theorem \ref{mainthm} is to the author's knowledge the first analysis that \textit{excludes} potential examples of surgeries on iterated torus knots bounding rational homology balls.

We note that only the ``only if" part of Theorem \ref{mainthm} is new, whereas the ``if" part follows from \cite[Theorem 1.3]{GALL}. (The reason only one of the two families of cables with positive surgeries bounding torus knots mentioned in \cite[Theorem 1.3]{GALL} appears is that the other family has surgery coefficient lower than $p_2\alpha_2$.) One may wonder if all surgeries on iterated torus knots that bound rational homology $4$-balls arise from \cite[Theorem 1.3]{GALL}, that is whether the following is true:

\begin{conj}
Suppose that $S^3_{n}(T(p_1,\alpha_1;p_2, \alpha_2))$ bounds a rational homology $B^4$. Then $p_2^2$ divides $n$, $S^3_{\frac{n}{p_2^2}}T(p_1, \alpha_1)$ bounds a rational homology ball, and $\alpha_2=\frac{n}{p_2} \pm 1$.
\end{conj}

Bodnár's examples \cite{bodnar} show that this is not true in general. However, the answer is unknown if we make the additional assumption that $n\geq 2+p_2\alpha_2$ and thus that $S^3_n(T(p_1,\alpha_1; p_2, \alpha_2))$ bounds a negative definite H-shaped plumbing of disc bundles over spheres. This is because only examples (iii) and (iv) in \cite[Theorem 3.1.1]{bodnar} give us an $S^3_n(T(p_1,\alpha_1; p_2, \alpha_2))$ satisfying the additional assumption, which both arise from \cite[Theorem 1.3]{GALL}. Bodnár's examples suggest that the finitely many integral surgery coefficients per knot for which the surgery does not bound a negative definite plumbing are the ones that are the most likely to give rise to a $3$-manifold that bounds a rational homology $4$-ball.

\subsection{Outline of Paper}
In Section \ref{latticeembeddings}, we give a brief introduction to lattice embeddings, directed at those new to the area, while establishing the notation and terminology. We also prove a basic proposition that we will use in Section \ref{latticeanalysissec}. In Section \ref{negdefsec}, we find plumbing diagrams for surgeries on algebraic iterated torus knots. In Section \ref{latticeanalysissec} we analyse which plumbing graphs admit lattice embeddings.

\subsection*{Acknowledgements}
Firstly, I would like to thank Ana G. Lecuona, my supervisor, for introducing me to Casson's problem, iterated torus knots and lattice embeddings. Secondly, I would also like to thank Marco Golla, who is not my supervisor, for always being there to answer my questions and expertly finding errors in my proofs. Thirdly, I would like to thank the anonymous referee for the helpful advice on the presentation. Last but not least, I would like to thank Laboratoire de Mathématiques Jean Leray in Nantes for hosting me in the spring of 2021 and providing me, despite the pandemic, with an office, mathematical discussions, friendly faces and happiness.

\subsection*{Competing Interests}
The author declares none.


\section{Preliminaries on Lattice Embeddings}
\label{latticeembeddings}

This section is to serve as a brief introduction to working with lattice embeddings. Recall Proposition \ref{donaldson}. In this paper, as well as many others, including \cite{aceto20, lecuonalisca11, liscasingle07, liscamultiple07,GALL}, we are working with $X$ a tree-shaped plumbing of disc bundles on spheres. Its second homology is the free abelian group $\Z\langle V_1,\dots,V_k \rangle $ on the vertices and the intersection form is \[ \langle V_i, V_j \rangle_{Q_{X}} =  \begin{cases}
    \text{weight of } V_i & \text{ if } i=j \\
    1 & \text{ if } V_i \text{ is adjacent to } V_j \\
    0 & \text{ otherwise.}
\end{cases} \] A \textbf{lattice embedding} $f: (H_2(X)/\Torsion, Q_{X}) \hookrightarrow (\Z^{\rk H_2(X)}, -\Id)$ is a homomorphism of abelian groups $f$ (by abuse of notation often called a linear map) such that $\langle V_i, V_j \rangle_{Q_{X}} = \langle f(V_i), f(V_j) \rangle_{-\Id}$. Sometimes we talk about lattice embeddings into other ranks, meaning that $f$ goes into $(\Z^r, -\Id)$ for some $r$ not necessarily equal to $\rk H_2(X).$ We denote $f(V_i)=v_i$. Usually, we mean $\langle \cdot , \cdot \rangle_{-\Id}$ when we write just $\langle \cdot , \cdot \rangle$. When we talk about \textbf{basis vectors}, we are referring to an orthonormal basis of $\Z^r$, that is the codomain. We say that a basis vector $e$ \textbf{hits} a vector $v_i$ if $\langle v_i, e \rangle \neq 0$. We say that a vector $w$ is \textbf{included} in $v$, or that $v$ \textbf{contains} $w$, if $v=w+u$ and there is no basis vector hitting both $w$ and $u$. We call a lattice embedding $f: (H_2(X)/\Torsion, Q_{X}) \to (\Z^r, -\Id)$ \textbf{essential} if every basis vector of the image hits some vertex.

\begin{figure}
    \centering
    \includegraphics{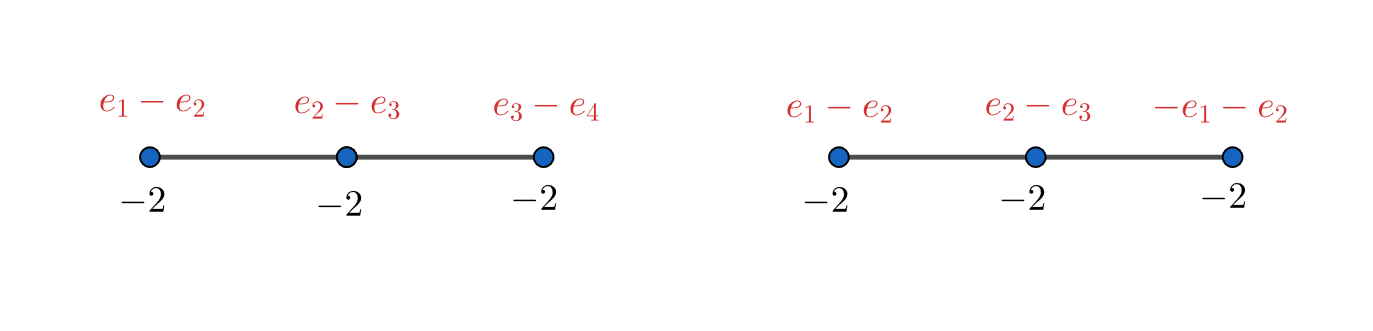}
    \caption{A $-2$-chain of length $3$ has two essential lattice embeddings, one into $(\Z^4, -\Id)$ and one into $(\Z^3, -\Id)$.}
    \label{fig:twoembeddings}
\end{figure}

When working with lattice embeddings of plumbing graphs, we often write the image of each vertex at the vertex, as for example in Figure \ref{fig:twoembeddings}. We consider two embeddings equivalent if they are the same up to signs and renaming of vertices, i.e. self-isometries of $(\Z^r, -\Id)$, and we only consider embeddings up to equivalence.

As an exercise in working with lattice embeddings, the reader is invited to prove the following standard fact:

\begin{prop}
\label{minustwochains}
A $-2$-chain of length $k \neq 3$ (that is a linear/path-shaped graph of length $k$ with all weights equal to $-2$) has a unique lattice embedding, essential into $(\Z^{k+1}, -\Id)$. A $-2$-chain of length $3$ has two lattice embeddings, shown in Figure \ref{fig:twoembeddings}.
\end{prop}

\begin{proof}[Proof sketch]
The left embedding in Figure \ref{fig:twoembeddings} easily generalises to an essential embedding of a $-2$-chain of length $k$ into $\Z^{k+1}$. Given an embedded graph, any subgraph has an induced embedding into some $(\Z^r, -\Id)$. Thus an embedding of $-2$-chain of length $k+1$ has to be an extension of the embedding of a $-2$-chain of length $k$. If $k\geq 3$, there is only one possible extension of an embedding like in the left part of Figure \ref{fig:twoembeddings}. The embedding in the right part cannot be extended at all.
\end{proof}

If the graphs $\Gamma_1$ and $\Gamma_2$ have embeddings into $(\Z^{k_1},-\Id)$ and $(\Z^{k_2},-\Id)$ respectively, then the disjoint union of the graphs has an induced embedding into $(\Z^{k_1+k_2}, -\Id)$, created by renaming the basis vectors of $\Z^{k_2}$ so that they are distinct from the basis vectors of $\Z^{k_1}$. The following corollary will be useful in Section \ref{latticeanalysissec}.

\begin{cor}
\label{unionofminus2chains}
An embedding of a disjoint union of $-2$-chains is, up to sign and renaming of vertices, a disjoint union of embeddings \begin{enumerate}
    \item $(e_1-e_2,...,e_k-e_{k+1})$,
    \item $(e_1-e_2,e_2-e_3,-e_1-e_2)$ and
    \item the embedding $(e_1-e_2),(e_1+e_2)$ of the two-component graph consisting of two disconnected vertices of weight $-2$.
\end{enumerate}
\end{cor}

\begin{proof}
Suppose that the embedding of the disjoint union of $-2$-chains has a component with embedding $(e_1-e_2,e_2-e_3,-e_1-e_2)$. Let $v=\lambda_1e_1+\lambda_2e_2+u$, where $u$ is not hit by $e_1$ or $e_2$, be the image of a vertex in a different component. By orthogonality to $e_1-e_2$, $\lambda_1=\lambda_2$, and by orthogonality to $-e_1-e_2$, $-\lambda_1=\lambda_2$. Thus $e_1$ and $e_2$ hit no other vertex. If $e_3$ hits another vertex, then $e_2$ must hit the same vertex by orthogonality to $e_2-e_3$, which cannot happen.

Suppose that the embedding of the disjoint union of $-2$-chains has a component with embedding $(e_1-e_2,\dots,e_k-e_{k+1})$ and one of the $e_i$'s for $1\leq i \leq k+1$ shows up again in a vertex with embedding $v$ in a different component. Since $\langle v, e_j-e_{j+1} \rangle_{-\Id} = 0$ for all $1\leq j \leq k$, $\lambda(e_1+\cdots+e_{k+1})$ is included in $v$ for some $\lambda \neq 0$, giving $v$ weight at most $-\lambda^2(k+1)$. Since $v$ has weight $-2$, $k+1=2$ and $\lambda^2=1$. So components that can share basis vectors must both have length 1, the embedding of a pair of such vertices sharing a basis vector being $(e_1-e_2),(e_1+e_2)$ up to sign and renaming. These basis vectors cannot occur in a third component by the same argument as above.
\end{proof}


\section{Plumbings Bounding Surgeries on Iterated Torus Knots}
\label{negdefsec}

Proposition \ref{donaldson} gives us an obstruction for a 3-manifold to bound a rational homology ball. In this paper we are interested in 3-manifolds of the form $S^3_n(T(p_1,\alpha_1;p_2,\alpha_2))$. Non-zero integral surgeries on knots in $S^3$ always bound a definite knot trace. 

\begin{notation*}
    Let $X$ be a $4$-manifold with boundary $S^3$, $K \subset \partial X$ a knot and $n$ an integer. Then $X_n(K)$ is the manifold obtained from attaching a $2$-handle to $X$ along $K$ with framing $n$. Especially when $X=D^4$, $X_n(K)$ is called the $n$-trace on $K$.
\end{notation*}

In particular, $Y=S^3_n(T(p_1,\alpha_1;p_2,\alpha_2))$ bounds $D^4_n(T(p_1,\alpha_1;p_2,\alpha_2))$, which has intersection form $n \Id_1$. The only restriction that the positive version of Proposition \ref{donaldson} provides us with is that $n$ be a square, whereas Aceto and Golla proved in \cite[Theorem 1.2]{acetogolla} that $Y$ will bound a rational homology ball for at most two positive $n$, making the first restriction seem futile. In this section we therefore find a different, negative definite, manifold $X$ that $Y$ bounds and whose intersection form is harder to embed.

The outline of this section is as follows. First, we use algebro-geometric facts to show that large integer surgeries on algebraic knots have a negative definite plumbing graph (Proposition \ref{prop:alggeoplumbingexistence}). Then, we use some recipes provided by Eisenbud and Neumann in \cite{eisenbudneumann} to explicitly describe the plumbing graphs of surgeries on iterated torus knots (Proposition \ref{aitork_plumbing}). Finally, for those iterated torus knots that are algebraic, we explicitly describe the plumbing graphs with the lowest possible positive index, which must thus be zero (Theorem \ref{thm:algknotsurgeryplumbing}). This section is largely based on a book by Eisenbud and Neumann (\cite{eisenbudneumann}) that provides several interesting recipes, including how to go from a singularity link to its splicing graph and from a splicing graph to a plumbing graph.

In \cite[Appendix to Chapter I]{eisenbudneumann}, Eisenbud and Neumann summarise what we know about singularities of plane curves, that is algebraic curves in $\CP^2$ or $\C^2$.

\begin{dfn*}
Let $f\in \C[x,y]$ be a non-zero polynomial vanishing at 0. Also let $C=V(f)=\{(x,y)\in \C^2 \ | \ f(x,y)=0\}$. The \textbf{singularity link} $L_\varepsilon \subset S^3_\varepsilon$ is the intersection of $C$ with a sphere $S^3_\varepsilon \subset \C^2$ centred at 0 and with sufficiently small radius $\varepsilon$.
\end{dfn*}

Singularity links are important because they describe plane curve singularities topologically. A plane curve singularity is topologically a wedge of discs, embedded inside a $4$-ball as the cone over the singularity link. If the plane curve has a self-intersection, then the singularity link at the self-intersection point has several components. A singular point that is not a self-intersection is called a \textbf{cusp} or a \textbf{cuspidal singularity}. Its singularity link is thus a knot.

\begin{dfn*}
An \textbf{algebraic knot} is a one-component singularity link, that is the link of a cuspidal singularity.
\end{dfn*}

It is interesting to know what links are singularity links. To describes the singularity link, we can try to solve $f(x,y)=0$ for $y$ in terms of $x$ around 0. If there is a singularity at 0, we cannot use the implicit function theorem and get $y$ as a function of $x$, but we can in fact find solutions in terms of fractional power series called Puiseux series. Each of these solutions describe a branch of the curve, and thus a component of the link. Different Puiseux series, differing by a change of variable $x \mapsto \zeta x$ for $\zeta$ a root of unity, describe the same branch. We can remove all but finitely many terms without changing the link until we get a minimal series of the form \[ y=x^{q_1/p_1}(a_1+x^{q_2/(p_1p_2)}(a_2+x^{q_3/(p_1p_2p_3)}(\cdots(a_{k-1}+x^{q_k/(p_1\cdots p_k)})\cdots) \] for pairs $(p_i,q_i)$ satisfying $p_i,q_i>0$ and $\gcd(p_i,q_i)=1$. (These pairs are called \textit{Newton pairs}.) Eisenbud and Neumann then show that the knot described by this Puiseux series is exactly $T(p_1,\alpha_1;p_2,\alpha_2;\cdots;p_k,\alpha_k)$ for $\alpha_1=q_1$ and $\alpha_{i+1}=q_{i+1}+p_{i+1}p_i\alpha_i$. We obtain the following alternative definition:

\begin{prop} 
A knot is algebraic if and only if it is an iterated torus knot \[ T(p_1,\alpha_1; \cdots ; p_k, \alpha_k) \] satisfying \begin{itemize}
    \item that $p_i, \alpha_i \geq 2$ for all $1 \leq i \leq k$, and
    \item that $\alpha_{i+1}>p_{i+1}p_i\alpha_i$ for all $1 \leq i \leq k-1$.
\end{itemize}
\end{prop} 

Using the algebro-geometric characterisation of these special iterated torus knots we may prove the following:

\begin{prop} \label{prop:alggeoplumbingexistence}
Let $K=T(p_1, \alpha_1; \dots ; p_k, \alpha_k)$ be an algebraic knot and $n \geq p_k\alpha_k+2$. Then $S_n^3(K)$ bounds a negative definite plumbing of disc bundles over spheres.
\end{prop}

\begin{proof}
Let $C \subset \C^2$ be a curve with singularity link $K$ at $0$. We may resolve the singularity using a sequence of blow-ups. In fact, by potentially blowing up a few more times, we can ensure that the reduced total inverse image is a simple normal crossing divisor. An example of this procedure can be seen in Figure \ref{fig:cuspresolution}.

\begin{figure}
\centering
\begin{subfigure}[b]{.45\linewidth}
\centering
\includegraphics[width=0.9\linewidth]{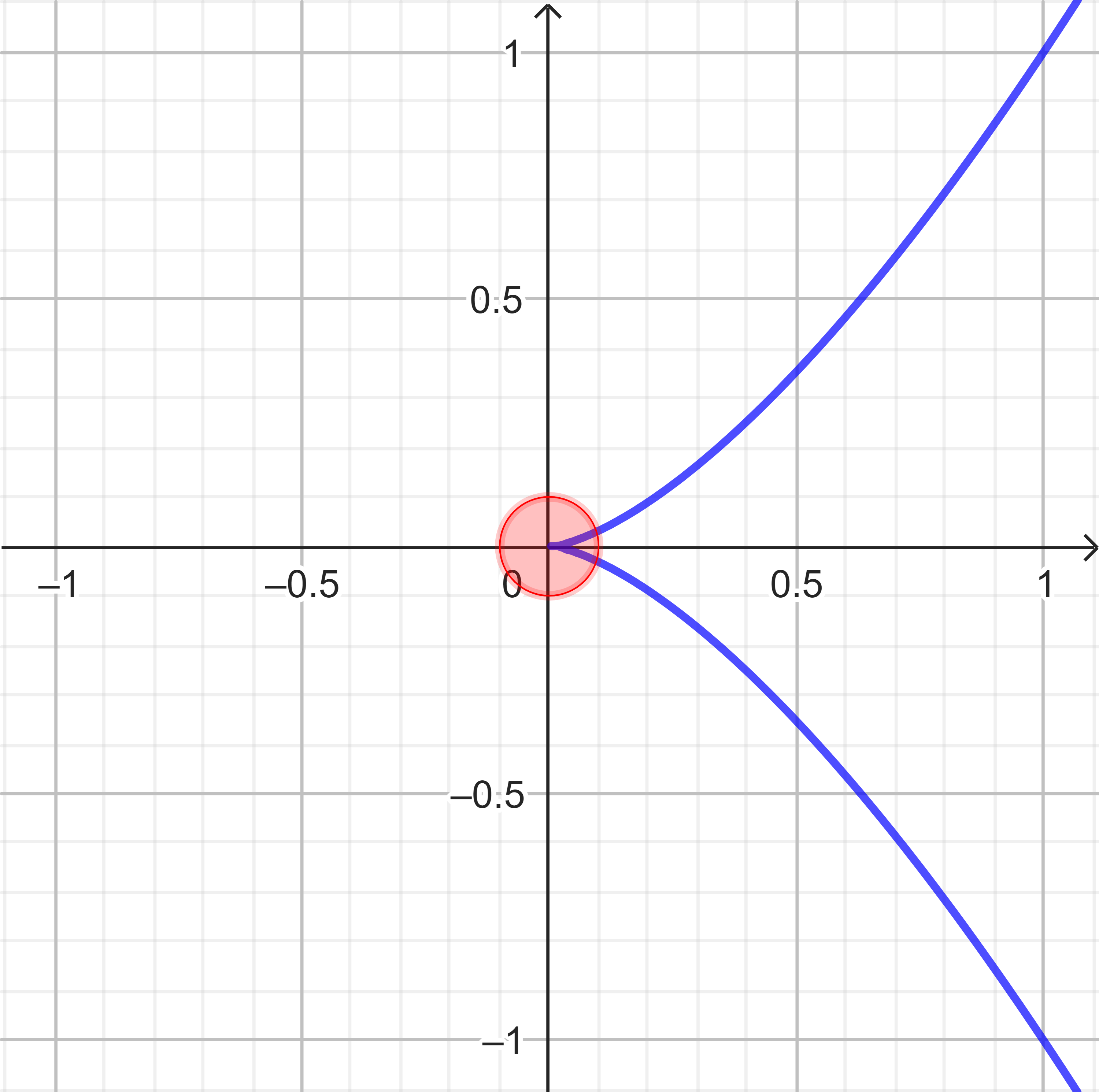}
\caption{}\label{fig:cuspresolution1}
\end{subfigure}
\begin{subfigure}[b]{.45\linewidth}
\centering
\includegraphics[width=0.9\linewidth]{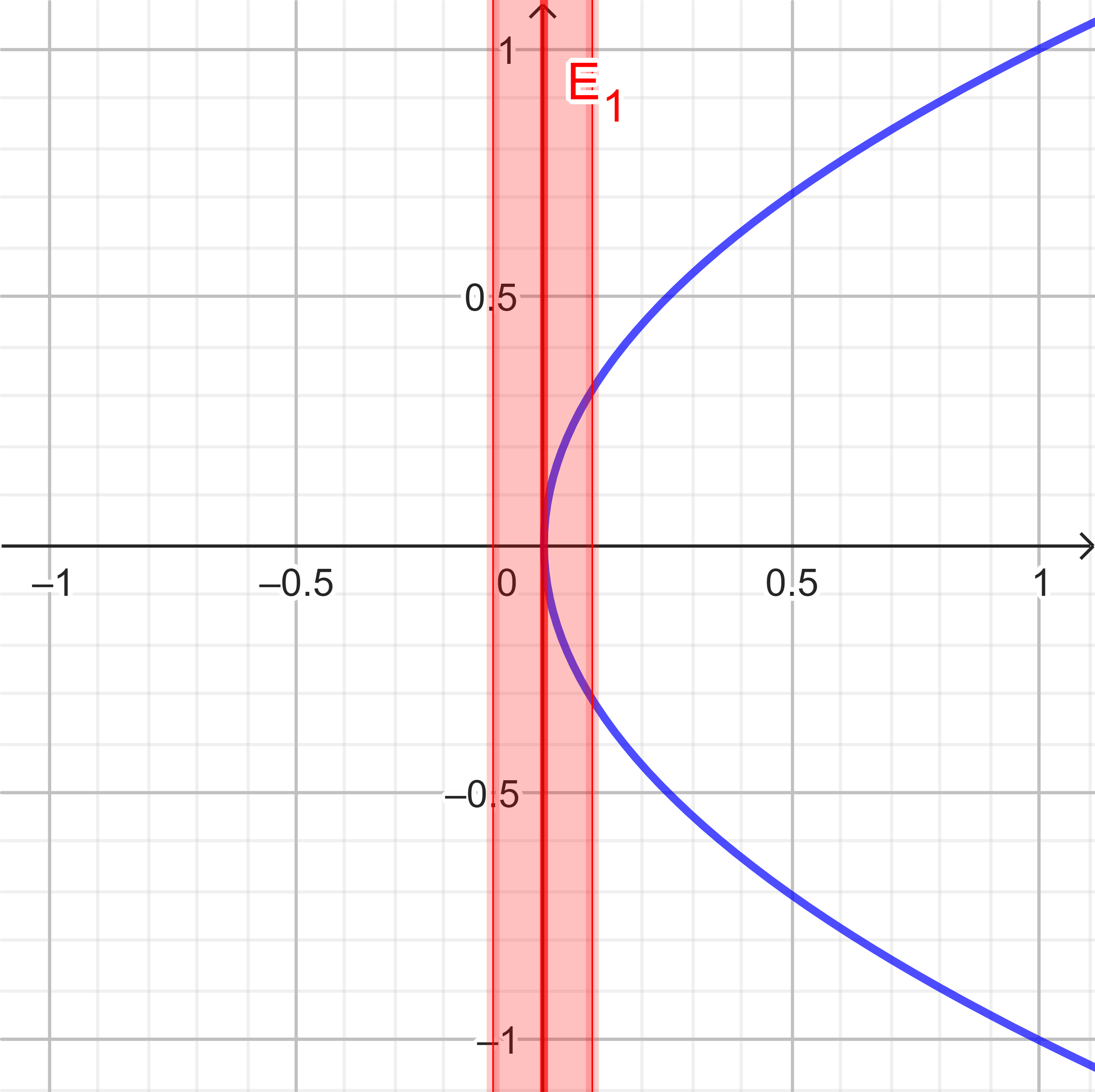}
\caption{}\label{fig:cuspresolution2}
\end{subfigure}

\begin{subfigure}[b]{.45\linewidth}
\centering
\includegraphics[width=0.9\linewidth]{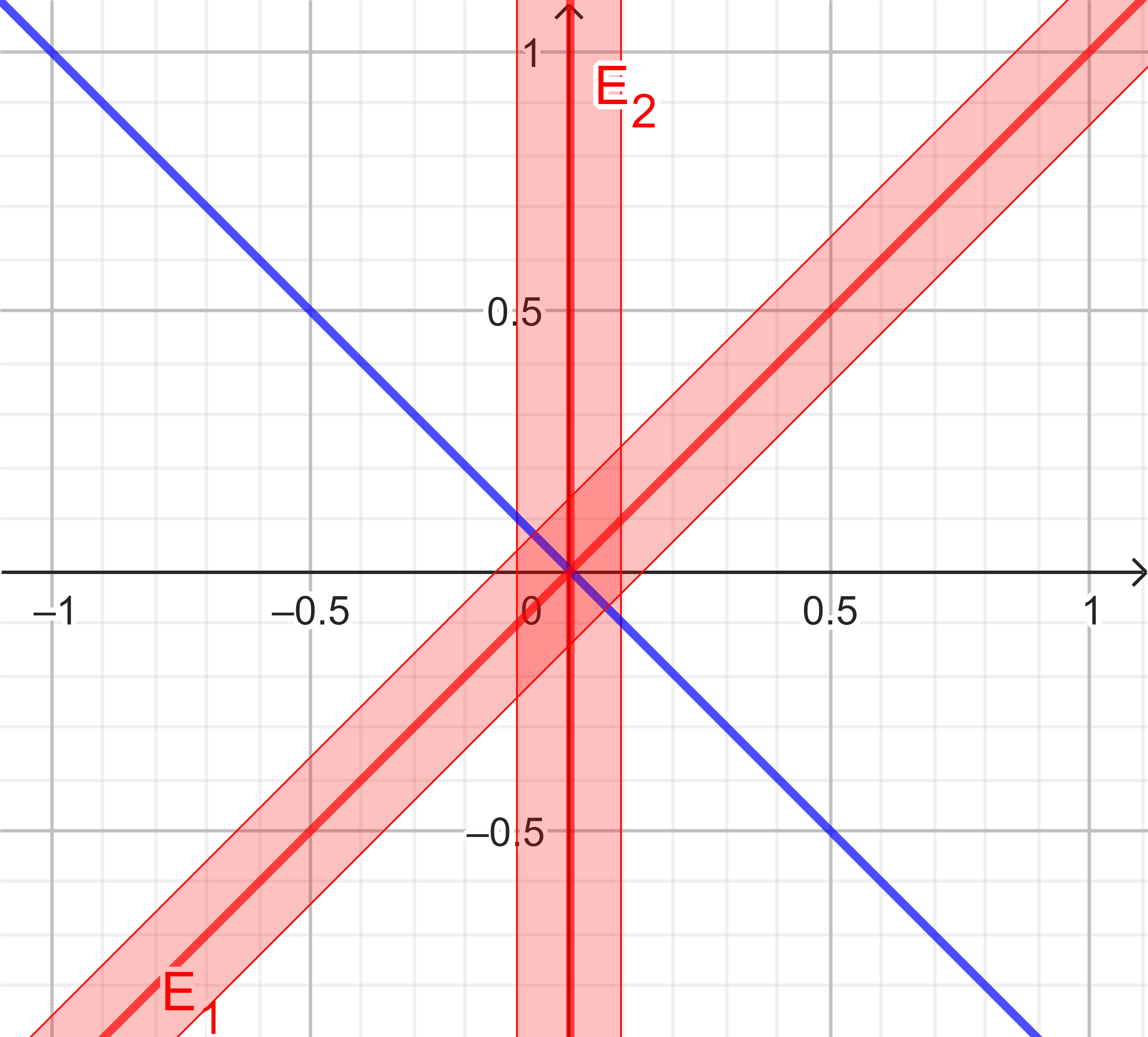}
\caption{}\label{fig:cuspresolution3}
\end{subfigure}
\begin{subfigure}[b]{.45\linewidth}
\centering
\includegraphics[width=0.9\linewidth]{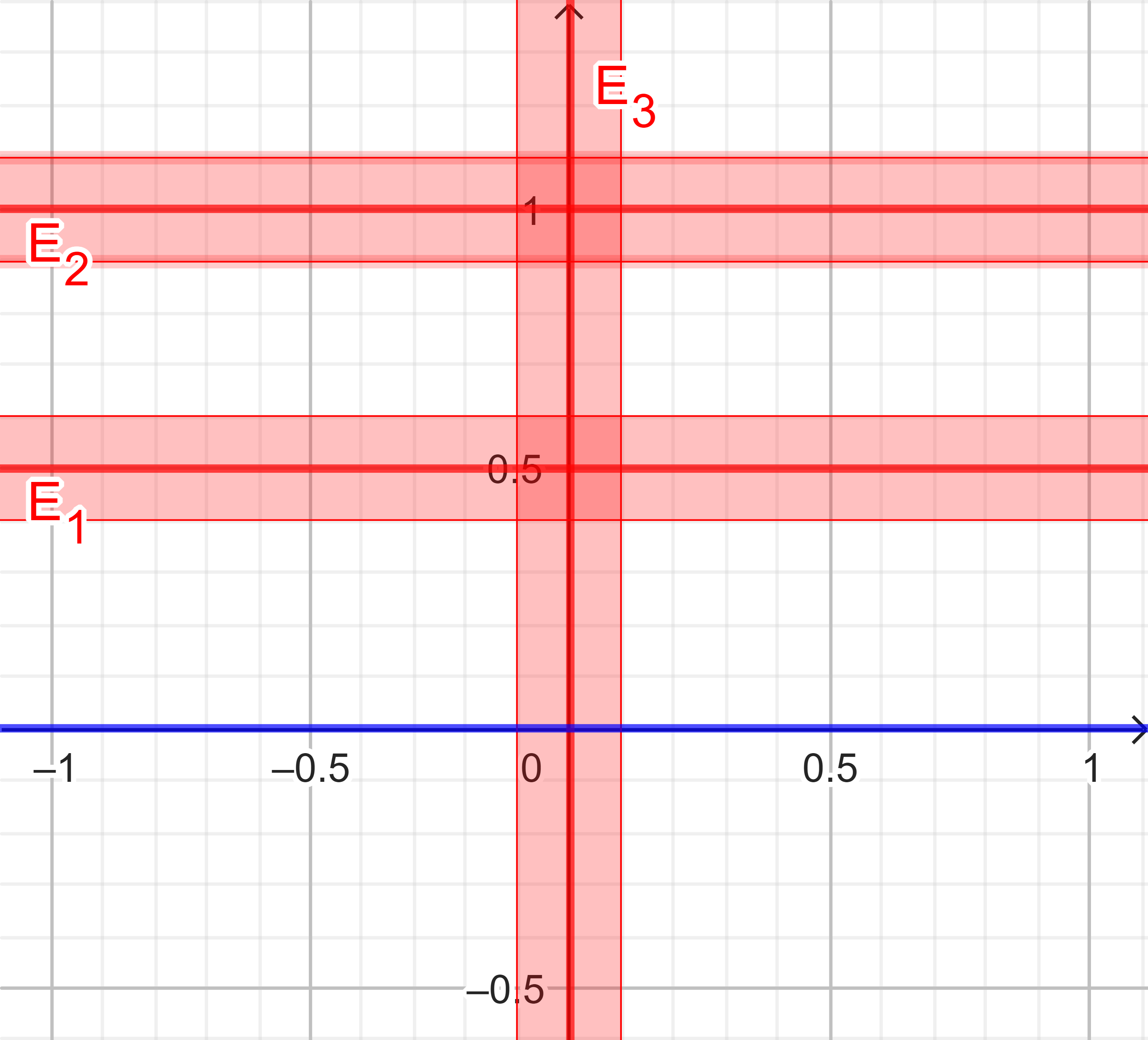}
\caption{}\label{fig:cuspresolution4}
\end{subfigure}
\caption{Minimal embedded resolution of the singularity $y^2=x^3$.}
\label{fig:cuspresolution}
\end{figure}

\begin{fact}
There exists a small 4-ball $D^4_\varepsilon \subset \C^2$ around 0, a complex surface $U \cong D^4_\varepsilon \# k \overline{\CP^2}$ and a map $p:U \to D^4_\varepsilon$ such that \begin{enumerate}
    \item $p|_{U-p^{-1}(0)}$ is a biholomorphic diffeomorphism onto $D^4_\varepsilon-\{0\}$,
    \item $p^{-1}(0)= E_1 \cup \cdots \cup E_k$ and $(f \circ p)^{-1}(0)= D \cup E_1 \cup \cdots \cup E_k$ where $D$ is a smooth disc called the proper transform of $C$ and $E_1,\dots ,E_k \cong \CP^1$ are smooth spheres called exceptional curves,
    \item $D,E_1,\dots ,E_k$ only have simple normal crossings, a.k.a. transverse double points, and
    \item the graph of intersections is in fact a tree with $D$ a leaf node.
\end{enumerate}
\end{fact}
\begin{rmk*} Such a singularity resolution is called an \textbf{embedded resolution}. \end{rmk*} Note that $\partial U=S^3$ still, with the singularity link intact inside it, but $X=U_n(K)$ has intersection form $n \Id_1 \oplus -\Id_k$, which if $n$ is positive means that the intersection form has positive index 1.

\begin{figure}[h]
\centering
\includegraphics[width=0.5\textwidth]{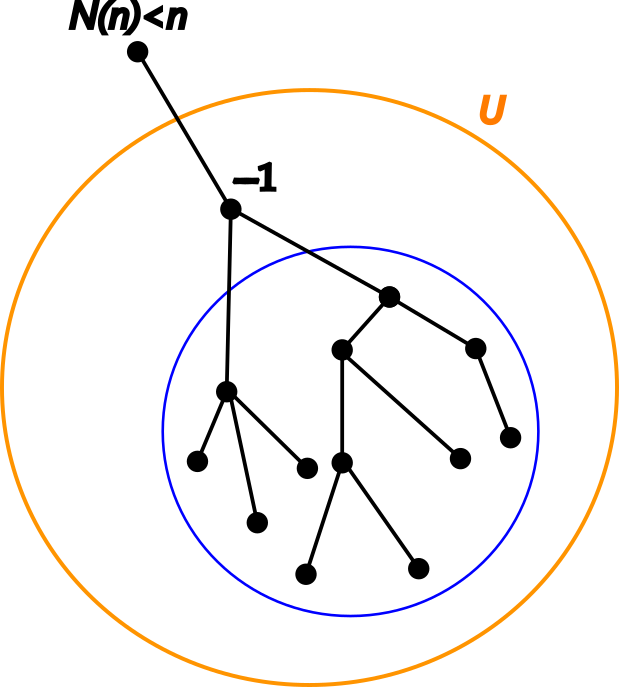}
\caption{The graph is a tree. All vertices inside the blue circle have weights at most $-2$. The 3-manifold represented by the piece of the graph inside the orange circle is $S^3$.}
\label{plumbinggraphfromalggeo}
\end{figure}

We know that $X$ has a Kirby diagram which is a disjoint union of an $n$-framed $K$ and $k$ unknots that are $-1$-framed, but we also have a more interesting representation. The way $D^4_\varepsilon$ was a regular neighbourhood of $0$, $U$ is a regular neighbourhood of $p^{-1}(0)= E_1 \cup \cdots \cup E_k$, with all crossings simple and normal. That makes $U$ a plumbing of sphere bundles, each sphere bundle being a neighbourhood of an $E_i$. Inside $U$ lies $D$, a smooth surface that intersects exactly one $E_i$ exactly once. (See Figure \ref{fig:cuspresolution4} for a great illustration. There $U$ is the red area, which is clearly a plumbing, and $D$ is the curve in blue.) Thus $X$ is a plumbing of sphere bundles obtained by adding one sphere bundle with some Euler number $N$ depending on $n$ to $U$'s plumbing representation. In fact, we know more. Every blow-up we make to create the minimal embedded resolution happens on the proper transform of the curve. Since a blow-up always decreases the self-intersection of a curve and the self intersection of the last exceptional curve we have added is always $-1$, the plumbing graph must have a shape as in Figure \ref{plumbinggraphfromalggeo}, that is a tree-shape, having one leaf of weight $N$ depending on $n$, one vertex of weight $-1$ and connected to the vertex of weight $N$ and the remaining vertices having weight no more than $-2$, and satisfying that the plumbing subgraph induced by excluding the vertex of weight $N$  has boundary $S^3$.

We can determine $N$. Note that the process of desingularising $C$ happens inside the small ball $D^4_\varepsilon$ around 0, and only depends on the singularity link, not on $n$. Blowing up decreases all self-intersections by a constant depending on where we blew up. Thus $N(n)=n-c$ for some positive constant $c$. Now, \cite[Theorem 18.3.4]{eisenbudneumann} tells us that the graph of Figure \ref{plumbinggraphfromalggeo} and the piece inside the blue circle have the same boundary if $N=0$, so if (and only if) $N=0$, $S^3_n(K)$ might be the connected sum of two different graph manifolds, which Gordon \cite{cameron} tells us in Theorem 7.5 only happens when $n=p_k\alpha_k$ for $k$ the index of the last Newton pair. Thus $N=n-p_k\alpha_k$.

\begin{figure}[h]
\centering
\includegraphics[width=\textwidth]{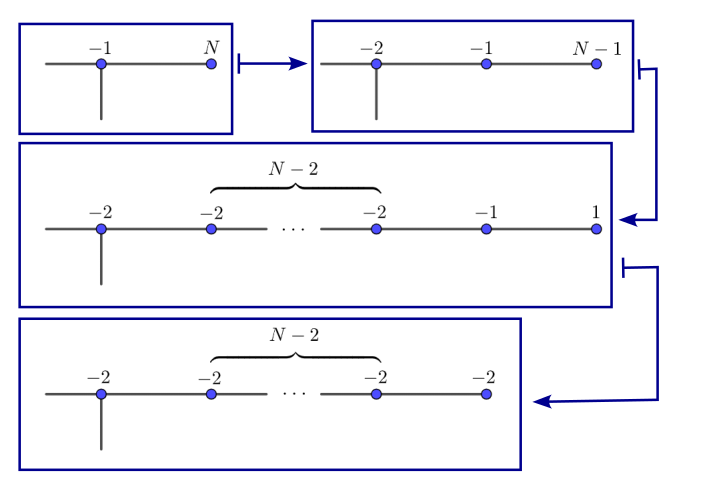}
\caption{Positive index-lowering transformation of $N$ into a sequence of $-2$'s.}
\label{Nto-2}
\end{figure}

Recall that when $n>0$, the $X$ we have described in Figure \ref{plumbinggraphfromalggeo} has positive index 1. If $N \geq 2$, we can use the sequence of blow-ups and blow-downs in Figure \ref{Nto-2} to transform it into a chain of $-2$'s and the $-1$ into a $-2$. Since in the process we blow down a $1$, this lowers the positive index by 1, giving us a negative definite graph. 
\end{proof}

\begin{rmk*}
If $N$ is negative, then $\partial X$ cannot bound a 4-manifold with a negative definite plumbing tree by \cite[Theorem 1.2]{neumann89} and Neumann's plumbing calculus (\cite[Theorem 3.2]{neumannplumbingcalculus}) since the graph in Figure \ref{plumbinggraphfromalggeo} is already in normal form and thus has the least positive index out of all plumbing trees. This is why we in Theorem \ref{mainthm} restrict ourselves to the case when $N \geq 2$, that is $n \geq p_2\alpha_2 +2$.
\end{rmk*}

Now we know exactly when surgeries on algebraic torus knots have negative definite plumbing graphs, but not exactly what these plumbing graphs look like. Fortunately, Eisenbud and Neumann have recipes for constructing them. In Chapter V of \cite{eisenbudneumann} they describe an algorithm for computing a plumbing representation of a 3-manifold from a splice diagram, which we have for $S^3_n(T(p_1,\alpha_1;p_2,\alpha_2))$ from Appendix to Chapter I. We would not need algebraicity in order to obtain such a graph, but we need algebraicity to obtain a nice simplification.

\begin{notation*}
    \[ [a_1,\dots, a_s]^-=a_1-\cfrac{1}{a_2-\cfrac{1}{\ddots - \cfrac{1}{a_s}}}. \]
\end{notation*}

\begin{figure}[h]
\centering
\includegraphics[width=\textwidth]{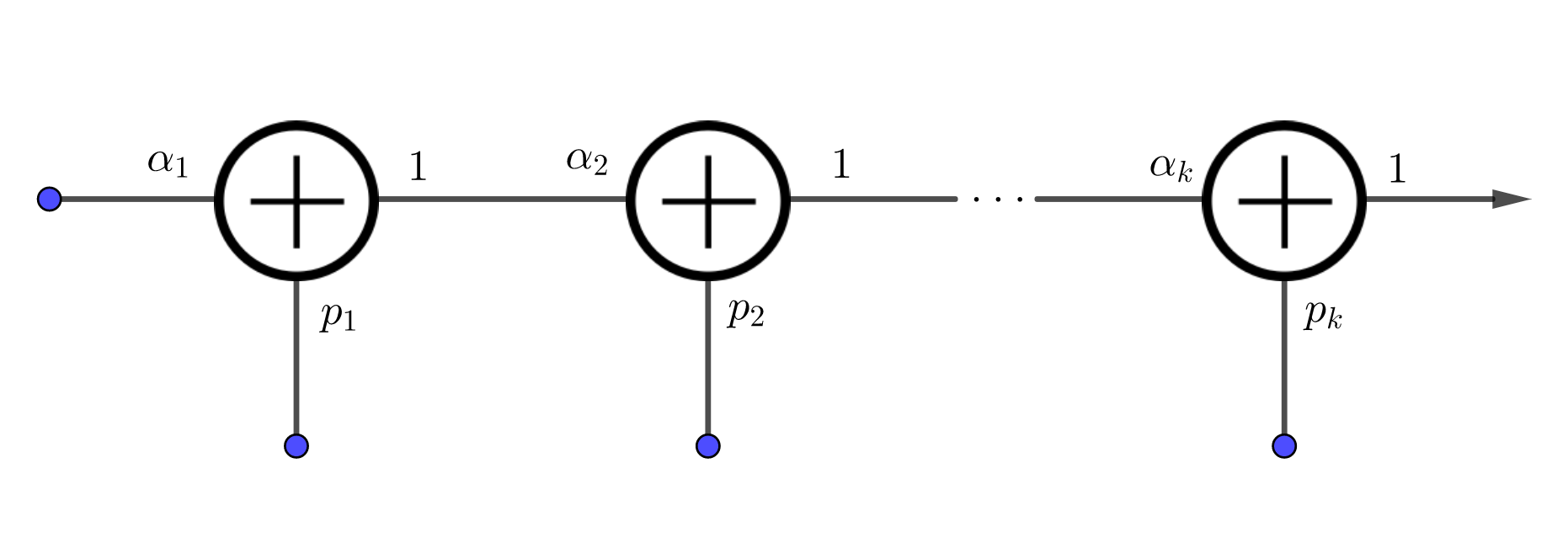}
\caption{Splicing diagram of an iterated torus knot.}
\label{itksplicinggraph}
\end{figure}

\begin{prop}
\label{aitork_plumbing}
Let $K=T(p_1,\alpha_1;p_2,\alpha_2;\dots ;p_k,\alpha_k)$ be a positive iterated torus knot. Then $S^3_n(K)$ bounds a 4-manifold with a plumbing graph as in Figure \ref{surgery-on-aitork-plumbing}, where $N=n-p_k\alpha_k$, and each $V_i$ is the graph in Figure \ref{vinkelhake}, with $[c_{i,2},c_{i,3},\dots , c_{i,s_i}]^{-}=\frac{\alpha_i}{\alpha_i-p_i}$ and $[d_{i,1},d_{i,2},\dots ,d_{i,t_i}]^-=\frac{\alpha_i}{p_i}$. Moreover, if $\alpha_i/p_i>2$, then $c_{i,2}=c_{i,3}=\cdots=c_{i,d_{i,1}-1}=2$.
\begin{figure}[h]
\centering
\includegraphics[width=\textwidth]{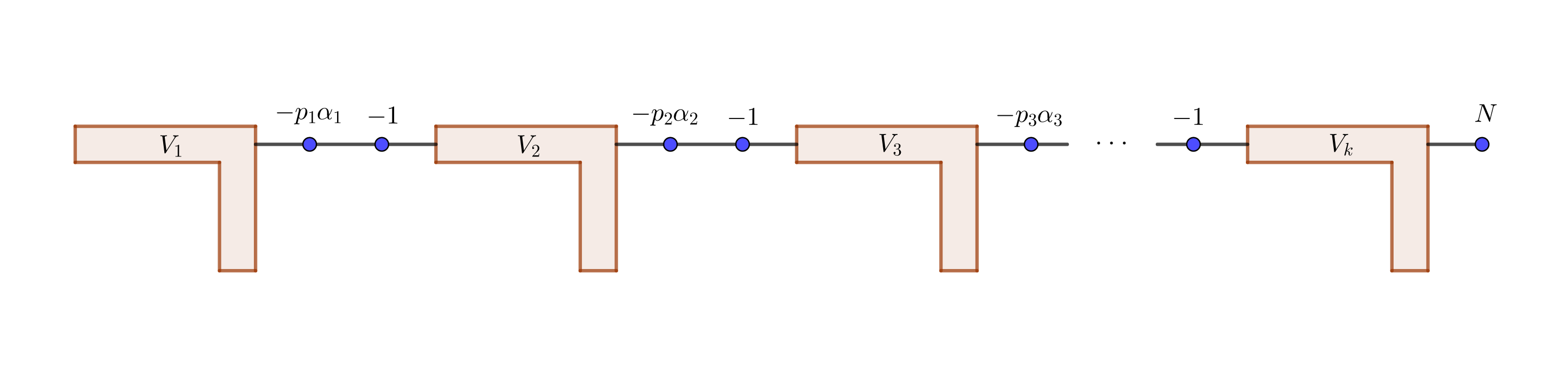}
\caption{Plumbing diagram of $S^3_n(T(p_1,\alpha_1;p_2,\alpha_2;\dots;p_k,\alpha_k))$.}
\label{surgery-on-aitork-plumbing}
\end{figure}

\begin{figure}[h]
\centering
\includegraphics[width=\textwidth]{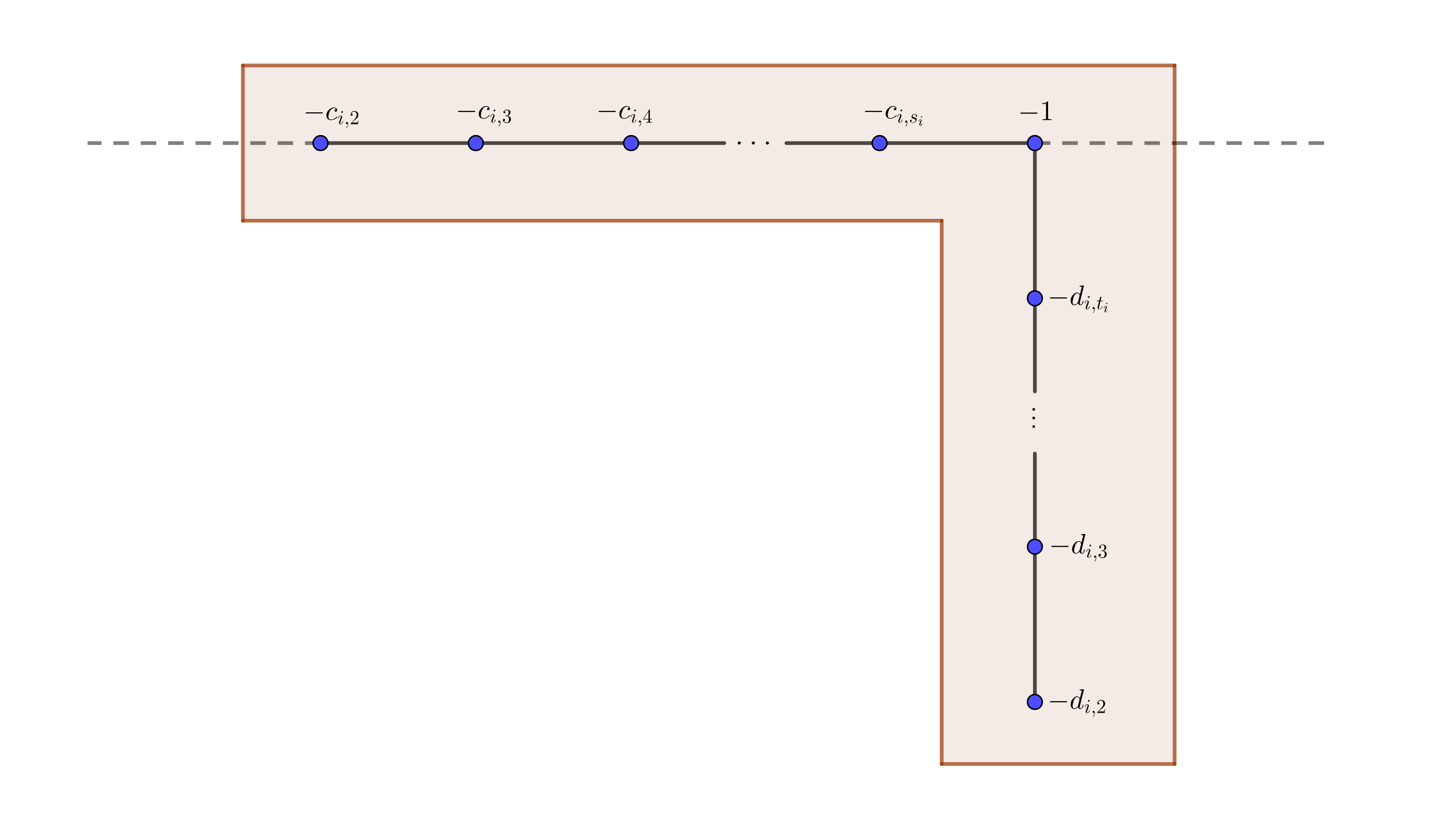}
\caption{Close-up diagram of each $V_i$ part of the plumbing diagram of $S^3_n(T(p_1,\alpha_1;p_2,\alpha_2;\dots;p_k,\alpha_k))$ in Figure \ref{aitork_plumbing}.}
\label{vinkelhake}
\end{figure}
\end{prop}

Note that $[c_{i,2},c_{i,3},\dots , c_{i,s_i}]^{-}$ and $[d_{i,1},d_{i,2},\dots ,d_{i,t_i}]^-$ are related by the Riemenschneider point rule (\cite[Section 3]{riemenschneider} for the original source, or \cite[Section 3]{lecuonalisca11} for an explanation in English). We may also point out that our unusual choice of indexing of the $c$'s is due to the fact that $[1,c_{i,2},c_{i,3},\dots , c_{i,s_i}]^{-}=\frac{p_i}{\alpha_i}$.

\begin{proof}
Chapter II and Appendix to Chapter 1 in \cite{eisenbudneumann} describe how to write down a cabling using a splicing graph, the result being as shown in Figure \ref{itksplicinggraph}. Section 22 in \cite{eisenbudneumann} then gives a recipe for translating this graph into a plumbing graph. First, each plus in the graph gets translated into the hook of Figure \ref{vinkelhake} using Theorem 22.1. The only non-trivial part here is computing the number at the corner of the hook. According to Theorem 22.1, it is the additive inverse of \[ \frac{1}{p_i\alpha_i}+ \frac{1}{[c_{i,s_i}, \dots , c_{i,2}]^-}+\frac{1}{[d_{i,t_i}, \dots , d_{i,2}]^-}=\frac{1}{p_i\alpha_i}+\frac{(\alpha_i-p_i)^*}{\alpha_i}+\frac{\left( p_i\left \lceil{\frac{\alpha_i}{p_i}}\right \rceil - \alpha_i\right)^* }{p_i},\] where $0<a^*<b$ in $\frac{a^*}{b}$ is a number such that $aa^* \equiv 1 \pmod{b}$. The above computation uses the fact that if $[a_1,\dots ,a_l]^-=a/b$ for $b<a$ then $[a_l,\dots ,a_1]^-=a/(b^*)$ from for example \cite[Lemma 2.4]{owensstrle12}. The number above is a positive integer less than 2 since $\frac{1}{p_i\alpha_i}+\frac{(\alpha_i-p_i)^*}{\alpha_i}<1$.

\begin{figure}[h]
\centering
\includegraphics[width=\textwidth]{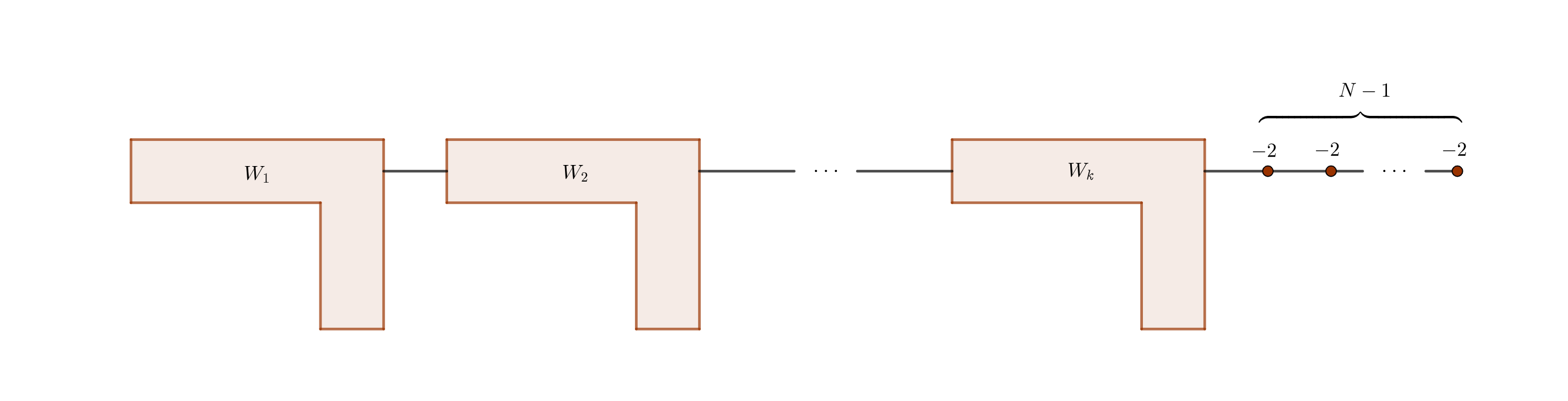}
\caption{The minimal negative definite plumbing diagram of $S^3_n(T(p_1,\alpha_1;p_2,\alpha_2;\dots;p_k,\alpha_k))$. Each $W_i$ is a subgraph described by Figure \ref{minimalvinkelhake}. Here $N=n-p_k\alpha_k$.}
\label{minimalplumbingaitork}
\end{figure}

Figure \ref{surgery-on-aitork-plumbing} (with an arrow instead of an $N$) is obtained from Theorem 22.2, the Addendum to Theorem 22.1 and the fact that $\left \lceil{\frac{p_i}{\alpha_i}}\right \rceil=1$. That picture represents a knot inside a plumbed 3-manifold (boundary of a plumbed 4-manifold), which in our case is $S^3$. Note that several different plumbings can represent the same 3-manifold, but they are related to each other through blow-ups and blow-downs and some other 0-related moves by Neumann's plumbing calculus, whose formulation adapted to links in plumbing graphs is \cite[Theorem 18.3]{eisenbudneumann}. Thus, Figure \ref{surgery-on-aitork-plumbing} (with an arrow instead of an $N$) is equivalent through these moves to the graph of $U$ in Figure \ref{plumbinggraphfromalggeo} with an arrow sticking out of the $-1$-vertex. We will see later exactly how to go from Figure \ref{surgery-on-aitork-plumbing} to the graph from algebraic geometry sketched in Figure \ref{plumbinggraphfromalggeo}. What is important is that the graphs are related by blow-ups and blow-downs happening away from the vertex of weight $N$, so the $N$'s of Figure \ref{surgery-on-aitork-plumbing} and Figure \ref{plumbinggraphfromalggeo} are the same.
\end{proof}

\begin{figure}[h]
\centering
\includegraphics[width=\textwidth]{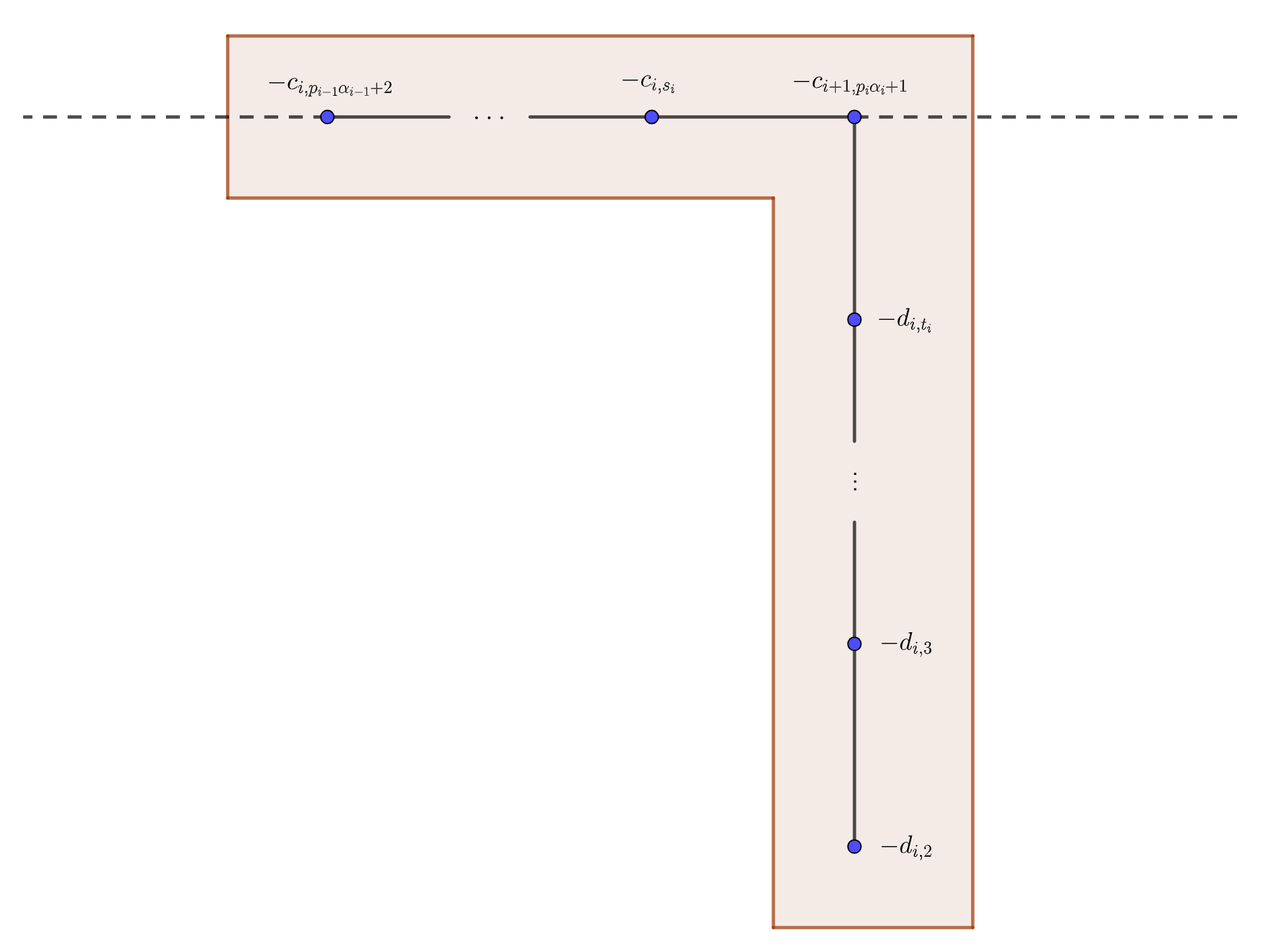}
\caption{Close-up diagram of each $W_i$ part of the plumbing diagram of $S^3_n(T(p_1,\alpha_1;p_2,\alpha_2;\dots;p_k,\alpha_k))$ in Figure \ref{minimalplumbingaitork}. In order for this diagram to make sense at the extremities, interpret $p_0\alpha_0$ as $0$, and $-c_{k+1,p_k\alpha_k+1}$ as $-2$.}
\label{minimalvinkelhake}
\end{figure}

The plumbing graph in Figure \ref{surgery-on-aitork-plumbing} is not ``minimal", that is, it contains vertices of weight $-1$ that can be blown down. It cannot be the one obtained from a blow-up resolution of a singularity since it is not of the form described in Figure \ref{plumbinggraphfromalggeo}. We note that if $K=T(p_1,\alpha_1;p_2,\alpha_2;\dots ;p_k,\alpha_k)$ is an algebraic knot, then $d_{i+1}=\big\lceil \frac{\alpha_{i+1}}{p_{i+1}}\big\rceil \geq p_i\alpha_i+1$, so the sequence $(c_{i+1,2},c_{i+1,3},\dots ,c_{i+1,s_i})$ is initiated by at least $p_i\alpha_i-1$ twos. There must also be a non-two in the sequence, as otherwise, by Riemenschneider's point rule, $\alpha_i/p_i$ would be an integer. Thus, for algebraic knots, we are going to strengthen Proposition \ref{aitork_plumbing} into the following theorem:

\begin{thm} \label{thm:algknotsurgeryplumbing}
    Let $K=T(p_1,\alpha_1;p_2,\alpha_2;\dots ;p_k,\alpha_k)$ be an algebraic knot and let $n \geq p_k\alpha_k+2$. Then $S^3_n(K)$ bounds a negative definite plumbed $4$-manifold with the graph shown in Figure \ref{minimalplumbingaitork}, where each hook $W_i$ is described by Figure \ref{minimalvinkelhake}. Again, $N=n-p_k\alpha_k$, $[c_{i,2},c_{i,3},\dots , c_{i,s_i}]^{-}=\frac{\alpha_i}{\alpha_i-p_i}$ and $[d_{i,1},d_{i,2},\dots ,d_{i,t_i}]^-=\frac{\alpha_i}{p_i}$.
\end{thm}

\begin{proof}
Assuming that $K$ is algebraic and thus that the sequence $(c_{i+1,2},c_{i+1,3},\dots ,c_{i+1,s_i})$ is initiated by at least $p_i\alpha_i-1$ twos, Figure \ref{contractinggraph} shows us a way to contract the plumbing graph, substituting it by one representing the same $3$-manifold. The first step is a sequence of $p_i\alpha_i-1$ $(-1)$-blow-downs, the second one is another $-1$-blow-down, and the third one is a $0$-absorption corresponding to Theorem 18.3.3 in \cite{eisenbudneumann}, or Section 1 in \cite{neumann89}, which also describes the effect of these operations on the index of the intersection form of the 4-manifold, in this case the effect being that both the positive and the negative index is decreased by one. These contractions, which happen far away from the arrow/$N$ vertex, allow us to change our plumbing to one where all vertices except maybe the one of weight $N$ and the one adjacent to it (of weight $-1$) are of weight at most $-2$, just like the graph in Figure \ref{plumbinggraphfromalggeo}. In fact, our graph looks like the graph in Figure \ref{minimalplumbingaitork} but with the rightmost node having weight $-1$ instead of $-2$ and just an $N$-weighted vertex instead of the rightmost chain of $N-1$ vertices of weight $-2$. Using that $N\geq 2$ and the plumbing calculus of Figure \ref{Nto-2}, we finally obtain Figure \ref{minimalplumbingaitork}. The fact that the graph in Figure \ref{minimalplumbingaitork} is negative definite follows from the existence of a negative definite plumbing graph, the fact that the graph in Figure \ref{minimalplumbingaitork} is in normal form (defined on \cite[Page 4]{neumann89}) and by the uniqueness theorems of plumbing graphs in normal form \cite[Theorem 1.2]{neumann89} and \cite[Theorem 18.3]{eisenbudneumann}.

\begin{figure}[h]
\centering
\includegraphics[width=\textwidth]{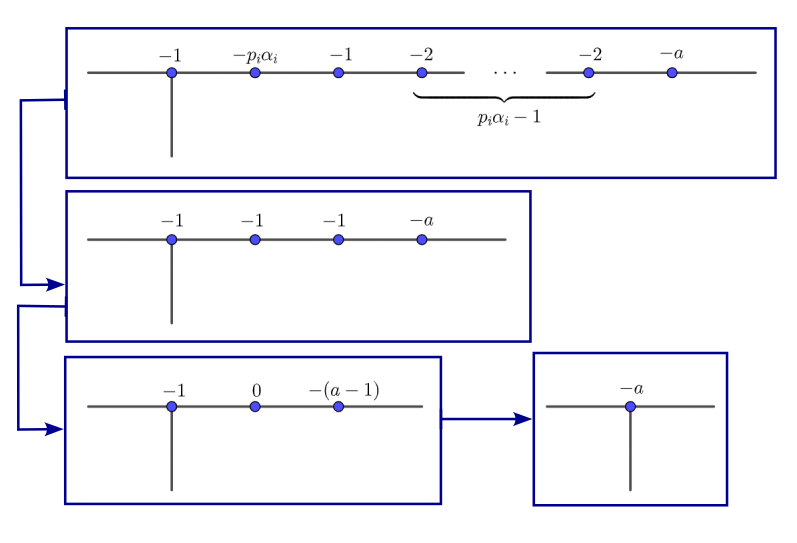}
\caption{Contracting the graph by blowing down}
\label{contractinggraph}
\end{figure}
\end{proof}

When working with lattice embeddings, we often need different tools depending on whether the trivalent vertices have weight $-2$ or if the weight is lower. In Section \ref{latticeanalysissec} this is relevant because if all trivalent vertices have weight $-2$, then none of the vertices with weight less that $-2$ are adjacent. On the other hand, if every vertex in a graph has weight at least its valency and one strictly greater, then its intersection form has a weakly chained diagonally dominant matrix, whose determinant is always non-zero. This can for example be used to argue that if such a graph has a lattice embedding, then a basis vector hitting only one vertex must be hitting a vertex with weight exactly its valency, something which has for example been used in \cite{GALL}. Because of this difference in available tools, we call the knot $T(p_1,\alpha_1;\dots ;p_k,\alpha_k)$ \textbf{super-algebraic} if all of the contractions in Figure \ref{contractinggraph} have $a=2$, that is $d_{i+1,1}-1=\left \lceil \frac{\alpha_{i+1}}{p_{i+1}} \right \rceil -1 \geq p_i\alpha_i$ for all $i$.

We end this section by introducing the following terminology, based on \textit{The Human Centipede}: considering the negative definite graph of Figure \ref{minimalplumbingaitork}, if we remove the trivalent vertices, the rightmost horizontal segment will be called \textbf{the tail}, the rest of them will be \textbf{torsos} numbered from left to right, and the vertical segments will be \textbf{legs}, also numbered from left to right. The trivalent vertices will simply be called \textbf{nodes}, also numbered from left to right. The union of Torso $i$, Leg $i$ and Node $i$ will be called \textbf{Body} $i$.


\section{Existence of Lattice Embeddings}
\label{latticeanalysissec}

In this section, we check our negative definite lattices from Section \ref{negdefsec} for embeddability in order to use the obstruction of Proposition \ref{donaldson}. The reader should be warned about the technical nature of lattice embeddings, and that the easiest way of understanding a proof using them is, just like for diagram chasing, to work it out on one's own.

We will first prove the theorem carefully in the super-algebraic case ($\alpha_2/p_2>p_1\alpha_1+1$ rather than $\alpha_2/p_2>p_1\alpha_1$). We split the proof for the super-algebraic case into two propositions depending on whether $\alpha_2 \equiv -1 \pmod{p_2}$ or $\alpha_2 \equiv 1 \pmod{p_2}$. The case $\lceil \alpha_2/p_2 \rceil = p_1\alpha_1+1$, proved in less detail as a separate proposition afterwards, is similar, but requires separate consideration due to some vertices of weight lower than $-2$ being adjacent.

\begin{prop}
\label{propminus1}
Let $\alpha_1 \equiv 1 \pmod{p_1}$, $\alpha_2 \equiv -1 \pmod{p_2}$, $\alpha_2/p_2 > p_1\alpha_1+1$ and $n \geq 2+p_2\alpha_2$. Then the rational homology 3-sphere $S^3_n(T(p_1,\alpha_1; p_2, \alpha_2))$ bounds a rational homology 4-ball if and only if the tuple \[ (p_1, \alpha_1; p_2, \alpha_2; n)\] is one of the following:
\begin{enumerate}
    \item $(p_1,p_1+1; p_2, p_2(p_1+1)^2-1; p_2^2(p_1+1)^2)$ or
    \item $(2,7; p_2, 16p_2-1; 16p_2^2)$.
\end{enumerate}
\end{prop}

\begin{figure}
\centering
\includegraphics[width=\textwidth]{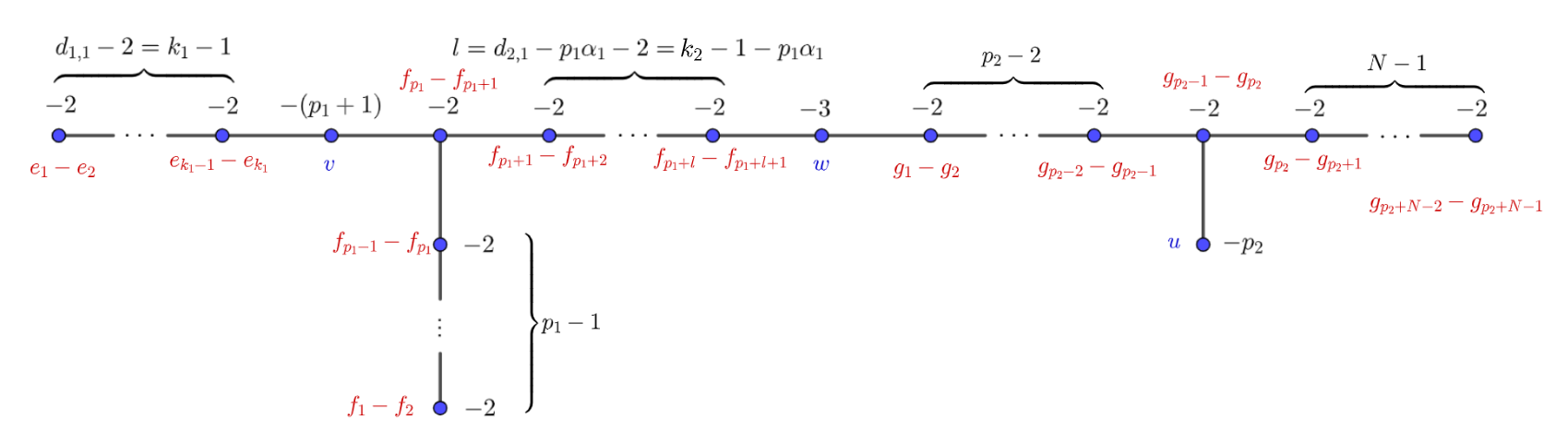}
\caption{Plumbing graph of $S_n^3(T(p_1,k_1p_1+1; p_2, k_2p_2+(p_2-1)))$ with $N \geq 2$ and $d_{2,1}=k_2+1 > p_1\alpha_1+1$.}
\label{alpha_2_cong_minus_1}
\end{figure}

\begin{proof}
We use the process described in Section \ref{negdefsec} to show that $S_n^3(T(p_1,k_1p_1+1; p_2, k_2p_2+(p_2-1)))$ bounds the plumbing in Figure \ref{alpha_2_cong_minus_1}, potentially with $k_1-1=0$. We do this by computing some negative continued fractions. First, we compute that $\alpha_1/p_1=[k_1+1,\underbrace{2,\dots,2}_{p_1-1}]^-$ and thus Leg 1 has weights $(\underbrace{-2,\dots ,-2}_{p_1-1})$ from bottom to top and by Riemenschneider's point rule, we get that Torso 1 has weights $(\underbrace{-2,\dots ,-2}_{k_1-1},-(p_1+1))$ from left to right. We have $\alpha_2/p_2=[k_2+1,p_2]^-$, leaving Leg 2 with one vertex of weight $-p_2$, whereas Riemenschneider's point rule gives a Torso 2 the weights $(\underbrace{-2, \dots , -2}_{k_2-1},-3,\underbrace{-2, \dots, -2}_{p_2-2})$ from left to right, but with the first $p_1\alpha_1$ $-2$'s cut off by the contraction described in Figure \ref{contractinggraph}. Note that if we only had algebraicity but not super-algebraicity, we would have the $-3$ in Node 1, adjacent to the vertex of weight $-(p_1+1)$.

By Corollary \ref{unionofminus2chains}, the embeddings of the $-2$-chains are forced as by Figure \ref{alpha_2_cong_minus_1}. Note that every $-2$-chain in the graph is extended at the end by some vertex not necessarily of weight $-2$, but we showed in the proof of Proposition \ref{minustwochains} that the embedding $(h_1-h_2,h_2-h_3,-h_1-h_2)$ can never be extended that way, which is why every $-2$-chain in Figure \ref{alpha_2_cong_minus_1} is embedded with one more basis vector than the number of vertices, no matter the length. We have three vertices left to embed, with embeddings $v$, $w$ and $u$, marked in blue in Figure \ref{alpha_2_cong_minus_1}. Note that there are at least 3 $g$'s, at least 3 $f$'s and either no or at least 2 $e$'s. For the obstruction of Proposition \ref{donaldson} to fail, we need to be able to embed the lattice of this graph in a lattice of the same rank as the number of vertices. If $k_1>1$, our partial embedding in red is already using as many basis vectors as we have access to. If $k_1=1$, we have access to one extra basis vector $h$.

There are three options for embedding $w$:
\begin{enumerate}
    \item $w=f_{p_1+l+1}+h-g_1$ implying that $k_1=1$,
    \item $w=f_{p_1+l+1}+g_2+g_3$ implying that $p_2=N=2$ and 
    \item $w=-g_1-f_1-f_2$ implying that $p_1=2$ and $l=0$.
\end{enumerate} We go through these cases one by one. \begin{enumerate}
    \item If $v$ is hit by $f_{p_1}$, then $v$ is hit by all vectors $f_1,\dots,f_{p_1}$. It must be hit by another vector, which must be $h$ since being hit by some $g$ means being hit by them all. However, that would mean $\langle v, w \rangle = \langle \pm h, h \rangle = \pm 1$, which is not the case. Thus $v$ is hit only by the $f$'s with index larger than $p_1$, which together with orthogonality to the all vectors but $u$ and the adjacent vertex, plus intersection 1 with the adjacent vertex, gives \[ v=(f_{p_1+1}+ \cdots + f_{p_1+l+1})-\lambda (g_1+\cdots +g_{p_2+N-1})-(1+\lambda)h \] for some $\lambda$.
    
    If $u$ is hit by $g_{p_2-1}$, it is also hit by all $g_1, \dots , g_{p_2-2}$. There is only space left for one basis vector, which must be $h$. The only possibility for orthogonality to $w$ becomes $u=-(g_1+\cdots+g_{p_2-1})-h$. Now $0=\langle v, u \rangle = - \lambda (p_2-1)-(1+\lambda)$. So $1=-\lambda p_2$, which is impossible since $p_2>1$. Thus $u$ is not hit by $g_{p_2-1}$, so it must be hit by $g_{p_2}$ and thus also $g_{p_2+1},\dots,g_{p_2+N-1}$. For orthogonality to the chains of $-2$'s and $w$, we get \[ u=g_{p_2}+\cdots + g_{p_2+N-1}+\kappa (f_1+\cdots + f_{p_1+l+1}-h).\] It remains to make sure that $\langle v, v \rangle = -(p_1+1)$, $\langle u, u \rangle = -p_2$ and $\langle v, u \rangle =0$. We get:  \[ \begin{cases}{}
     -(p_1+1) & =-(l+1)-\lambda^2(N+p_2-1)-(\lambda+1)^2\\
     -p_2 & =-N-\kappa^2(l+p_1+2)\\
     0 & =N\lambda-\kappa(l+1)-\kappa(\lambda+1)
\end{cases} \] Simplifying yields:
\begin{numcases}{}
    p_1=l+\lambda^2(N+p_2-1)+(\lambda+1)^2 \label{minuseq1}\\
    p_2=N+\kappa^2(l+p_1+2) \label{minuseq2}\\
    0=\lambda N - \kappa(l+\lambda+2)\label{minuseq3}
\end{numcases}Now, if $\kappa\neq 0$, Equation \eqref{minuseq2} implies that $p_2>p_1$. If that is so, Equation \eqref{minuseq1} gives that $\lambda=0$. Then Equation \eqref{minuseq3} gives that $\kappa=0$, which is a contradiction. Thus $\kappa=0$, implying though Equation \eqref{minuseq3} that $\lambda=0$, though Equation \eqref{minuseq2} that $p_2=N$, and through Equation \eqref{minuseq1} that $l=p_1-1$. This solution corresponds to \[(p_1,\alpha_1,p_2,\alpha_2,n)=(p_1,p_1+1, p_2, p_2(p_1+1)-1, p_2^2(p_1+1)^2)\] and is known to bound a rational homology ball by Theorem 1.3 in \cite{GALL}.

\item If $w=f_{p_1+l+1}+g_2+g_3$, then $u$ cannot be hit by $g_2$ as that would mean that $u=g_2+g_3$ and $\langle w, u \rangle=-2$. Thus $u$ is hit by $g_1$ and another basis vector. However, since the $e$'s and the $f$'s come in a package deal due to sitting in $-2$-chains orthogonal to $u$, none of these can hit $u$ and we have $k_1=1$ and $u=h-g_1$. As in the case above, if $f_{p_1}$ hits $v$, then $v=-(f_1+\cdots + f_{p_1}) \pm h$, but then $v$ intersects $u$. Thus $v=f_{p_1+1}+\cdots + f_{p_1+l+1}+\cdots$, through the $-2$-chain, but through orthogonality to $w$, $v$ must also be hit by one of $g_2$ and $g_3$. Through orthogonality to the tail vertex, the $g_2$ and $g_3$ must have the same coefficient in $v$, and $v$'s orthogonality to $w$ says that the coefficient of $f_{p_1+l+1}$ must be minus the double of the coefficient of $g_2$, which is impossible. Thus there are no embeddings with $w=f_{p_1+l+1}+g_2+g_3$.
\item Suppose $w=-g_1-f_1-f_2$, $p_1=2$ and $l=0$. Then $\langle v,v \rangle=-3$, so if $v$ is hit by $f_2$, it is hit by $f_1$ too with equal coefficient, and orthogonality to $w$ is impossible. Thus $v$ is hit by $f_3$ and two other basis vectors. Since there are at least 3 $g$'s, these must be $e$'s. We get $k_1=3$ and $v=-e_1-e_2+f_3$. Now $u$ cannot be hit by $g_{p_2-1}$ as it would be hit by all $g_1,\dots,g_{p_2-1}$ and there would be only one space left for another basis vector, whereas the $e$'s and $f$'s come in packages of 3. Thus $u=g_{p_2}+\cdots + g_{p_2+N-1}+\lambda(e_1+e_2+e_3)+\kappa(f_1+f_2+f_3)$. Orthogonality of $u$ to $w$ gives $\kappa=0$ and orthogonality of $u$ and $v$ then gives $\lambda=0$. This solution corresponds to $(p_1,\alpha_1,p_2,\alpha_2,n)=(2,7, p_2, 16p_2-1, 16p_2^2)$, which bounds a rational homology ball by Theorem 1.3 in \cite{GALL}.
\end{enumerate}
\end{proof}

Now we consider the case where $\alpha_2 \equiv 1 \pmod{p_2}$ instead. The reader may note that the only tuples $(p_1, \alpha_1; p_2, \alpha_2; n)$ for which we find embeddings in this case have $p_2=2$, in which case $\alpha_2 \equiv -1 \pmod{p_2}$ if and only if $\alpha_2 \equiv 1 \pmod{p_2}$.

\begin{prop}
\label{prop1}
Let $\alpha_1 \equiv 1 \pmod{p_1}$, $\alpha_2 \equiv 1 \pmod{p_2}$, $\alpha_2/p_2 > p_1\alpha_1+1$ and $n \geq 2+p_2\alpha_2$. Then the rational homology 3-sphere $S^3_n(T(p_1,\alpha_1; p_2, \alpha_2))$ bounds a rational homology 4-ball if and only if \[(p_1, \alpha_1; p_2, \alpha_2; n)=(2,7;2,31;64).\] or \[(p_1, \alpha_1; p_2, \alpha_2; n)=(p_1,p_1+1; 2, 2(p_1+1)^2-1; 4(p_1+1)^2).\]
\end{prop}

\begin{figure}
\centering
\includegraphics[width=\textwidth]{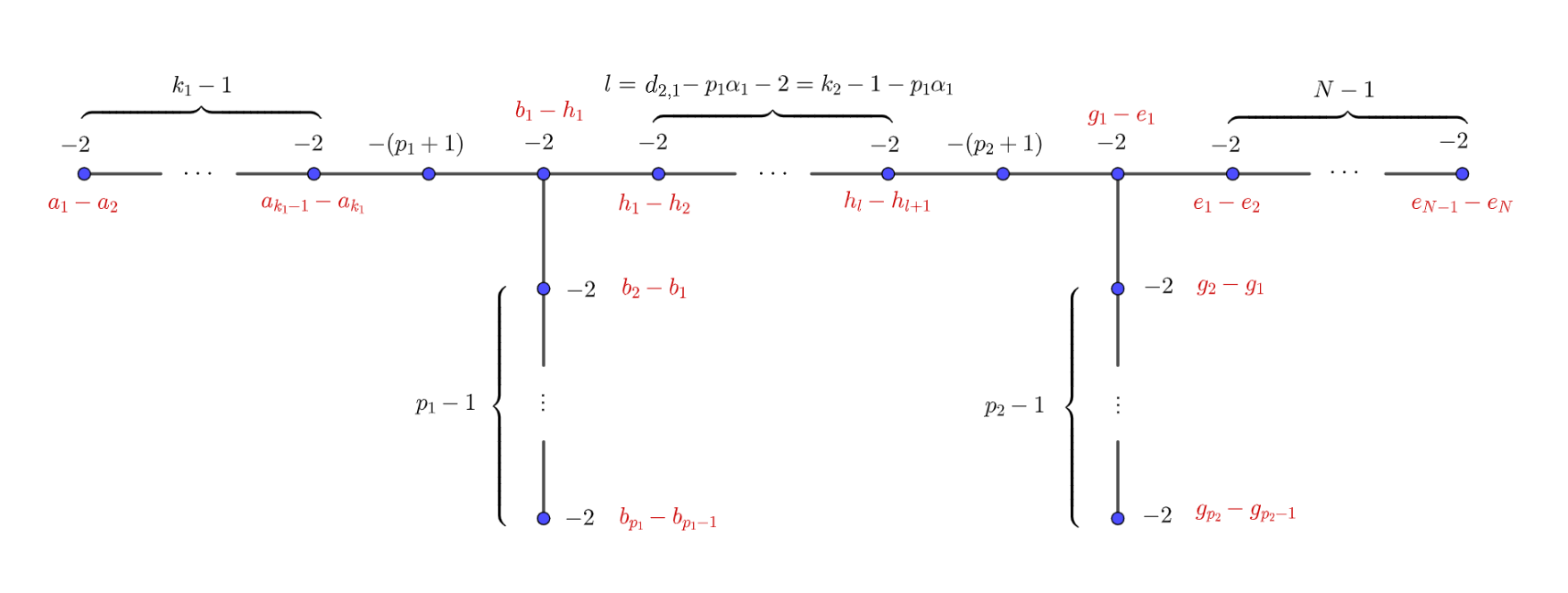}
\caption{Plumbing graph of $S_n^3(T(p_1,k_1p_1+1; p_2, k_2p_2+1))$ with $N \geq 2$ and $d_{2,1}=k_2+1 > p_1\alpha_1+1$.}
\label{alpha_2_cong_1}
\end{figure}

\begin{figure}
\centering
\includegraphics[width=\textwidth]{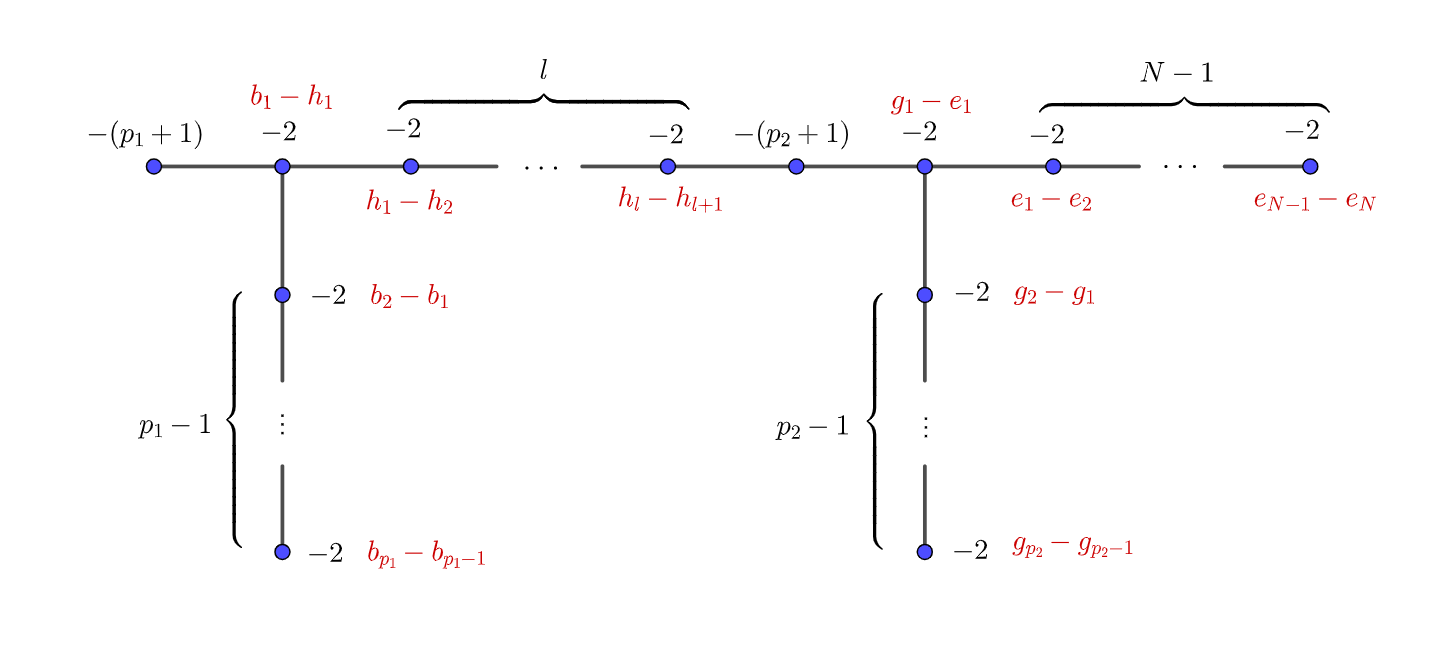}
\caption{Plumbing graph of $S_n^3(T(p_1,p_1+1; p_2, k_2p_2+1))$ with $N \geq 2$ and $d_{2,1}=k_2+1 > p_1\alpha_1+1$.}
\label{alpha_2_cong_1_and_p_1_plus_1}
\end{figure}

\begin{proof}
Since $\alpha_i/p_i=[k_i+1,\underbrace{2,\dots,2}_{p_i-1}]^-$, the recipe in Section \ref{negdefsec} gives us a plumbing graph with $(p_i-1)$ $-2$'s in each leg. Riemenschneider's point rule gives \[ [c_{i,2},c_{i,3},\dots, c_{i,s_i}]^{-}=[\underbrace{2,\dots,2}_{k_i-1},p_i+1]^- \] and the contraction in Figure \ref{contractinggraph} shortens the Torso by $(p_1\alpha_1-1)$ $-2$'s, leaving us with the negative definite graph in Figure \ref{alpha_2_cong_1}. By Proposition \ref{minustwochains}, a sequence of $-2$'s has only one possible embedding up to signs and renaming of elements, unless it has length 3, in which case there is an embedding of the form $(v_1-v_2,v_2+v_3,-v_1-v_2)$, but this embedding cannot be extended at its ends to a longer chain. Thus, unless $N=p_2=2$, the embedding of the $-2$-chains in Figure \ref{alpha_2_cong_1} must be the one in red. When $p_2=2$, $\alpha_2 \equiv 1 \pmod{p_2}$ if and only if $\alpha_2 \equiv -1 \pmod{p_2}$, a case that we have already dealt with completely. Hence we assume that $p_2>2$ and the red partial embedding is forced.

If $k_1-1>0$, we have already used more basis vectors than we have vertices in the graph, and thus no embedding satisfying the requirements of Proposition \ref{donaldson} can exist. If $k_1=1$, our graph looks like in Figure \ref{alpha_2_cong_1_and_p_1_plus_1}. We have two vertices left to embed: $v$ in Torso 1 and $w$ in Torso 2. It is easy to see that $v$ cannot be hit by $b_1$, since that would force $v=-(b_1+\cdots +b_{p_2})+u$ for some basis vector $u$, but there is no option for what $u$ could be. Thus $v=h_1+\cdots +h_{l+1}+\lambda(e_1+\cdots +e_{N}+g_1+\cdots + g_{p_2})$ for some $\lambda$. Now if $w$ is hit by $g_1$, then $w=-(g_1+\cdots+g_{p_2})+h_{l+1})$ by orthogonality to Leg 2 and having to hook on to the $-2$-chain in Torso 2. But then $0=\langle w, v \rangle =-1+\lambda p_2$, which is impossible. Hence, our embedding must be of the form described in Figure \ref{on_T(n,n+1)}.

\begin{figure}[h]
\centering
\includegraphics[width=\textwidth]{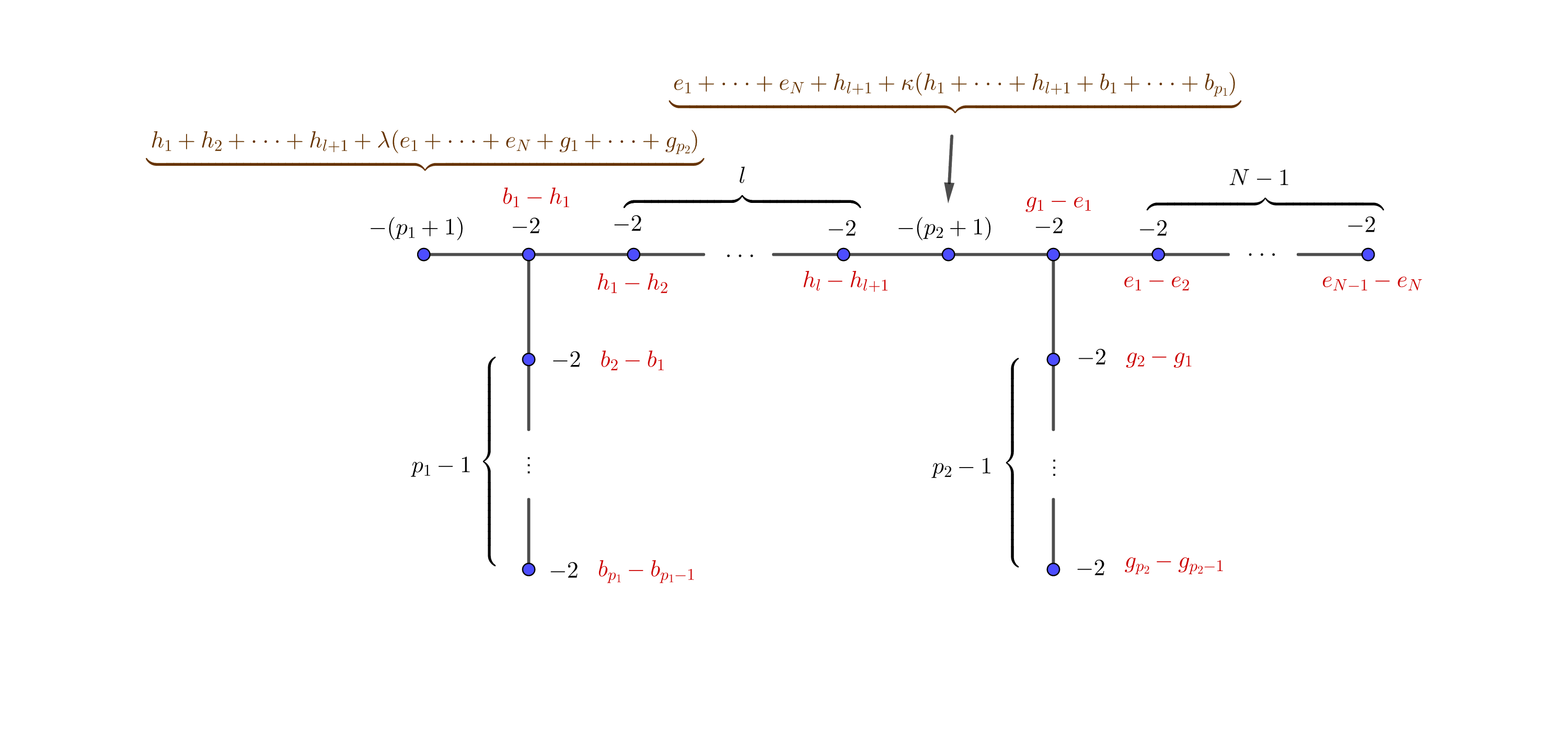}
\caption{Embedding of $S_n^3(T(p_1,p_1+1;p_2,kp_2+1))$}
\label{on_T(n,n+1)}
\end{figure}

There are 3 equations left to satisfy in order to obtain an embedding, determined by the relationship between the vertices of weight $-(p_1+1)$ and $-(p_2+1)$, and their relationship to themselves. These equations are:
 \[ \begin{cases}{}
    & -(p_1+1)=-(l+1)-\lambda^2(N+p_2)\\
    & -(p_2+1)=-N-(\kappa+1)^2-\kappa^2(l+p_1)\\
    & 0=-\kappa l - (\kappa+1)-\lambda N
\end{cases} \]
Simplifying them yields:
\begin{numcases}{}
    p_1=l+\lambda^2(N+p_2) \label{eq1}\\
    p_2=(N-1)+(\kappa+1)^2+\kappa^2(l+p_1) \label{eq2}\\
    0=\kappa l + (\kappa+1)+\lambda N \label{eq3}
\end{numcases}
Suppose $e_1$ shows up again in Body 1, meaning that $\lambda^2 \geq 1$. By Equation \eqref{eq1}, $p_1>p_2$. Now, if $\kappa^2\geq 1$, then by Equation \eqref{eq2}, $p_2>p_1$, which is a contradiction. However, if $\kappa=0$, then Equation \eqref{eq3} degenerates into $\lambda N = -1$, which contradicts that $N \geq 2$. If $e_1$ does not show up again in Body 1, meaning that $\lambda=0$, then Equation \eqref{eq1} gives that $l=p_1$ and Equation \eqref{eq3} gives that $-1=\kappa(l+1)=\kappa(p_1+1)$, which is impossible. Thus, the only embeddable cases are when $p_2=2$, in which case we get the cases coming from $\alpha_2 \equiv -1 \pmod{p_2}$.
\end{proof}

\begin{prop}
Let $\alpha_1 \equiv 1 \pmod{p_1}$ and $\alpha_2 \equiv \pm 1 \pmod{p_2}$. Also let $\lceil \alpha_2/p_2 \rceil =p_1 \alpha_1 +1$ and $n \geq 2+p_2\alpha_2$. Then $S_n^3(T(p_1,\alpha_1;p_2,\alpha_2))$ does not bound a rational homology 4-ball.
\end{prop}

\begin{figure}[h]
\centering
\includegraphics[width=\textwidth]{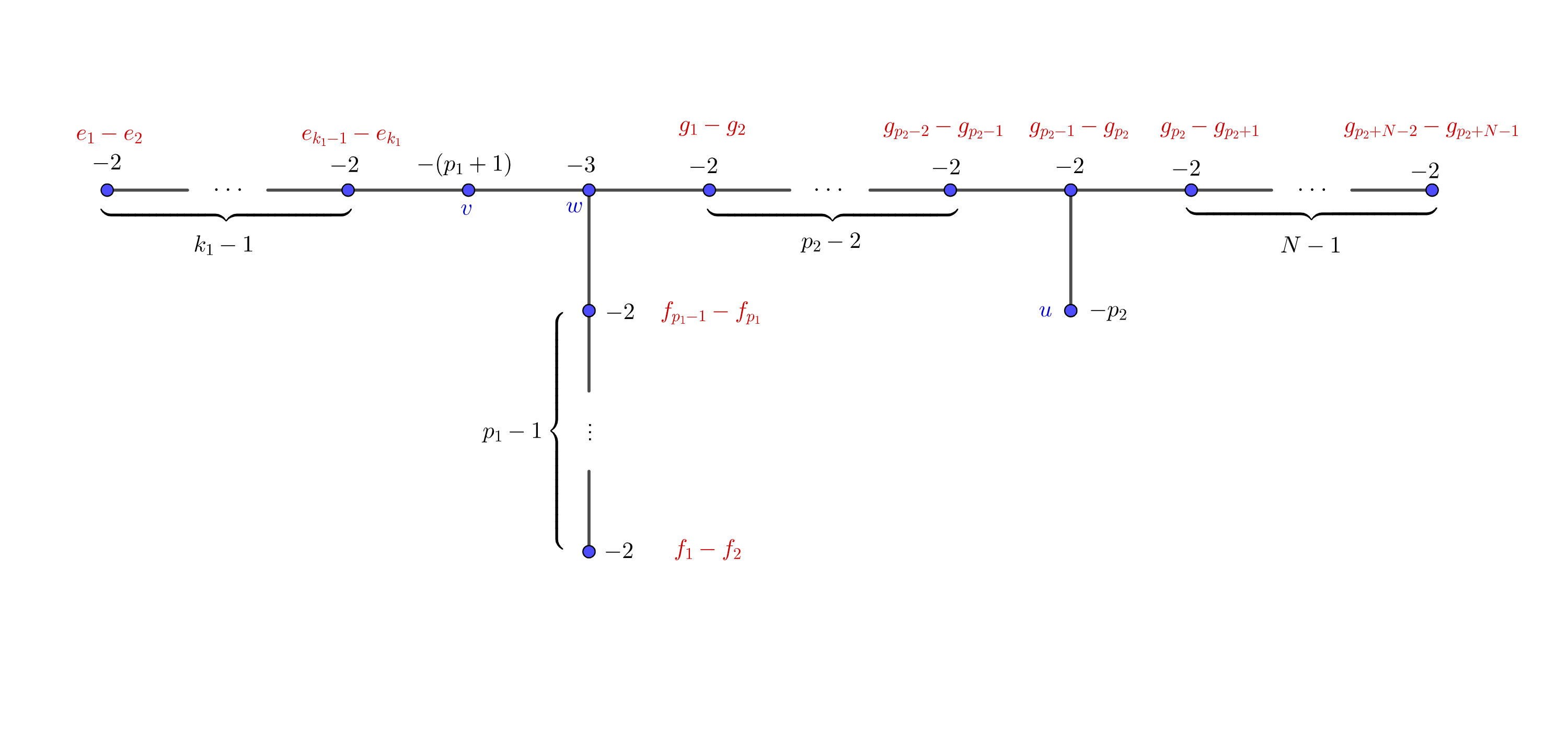}
\caption{Embedding of $S_n^3(T(p_1,k_1p_1+1;p_2,(p_1\alpha_1+1)p_2-1)$}
\label{algebraic1}
\end{figure}

\begin{figure}[h]
\centering
\includegraphics[width=\textwidth]{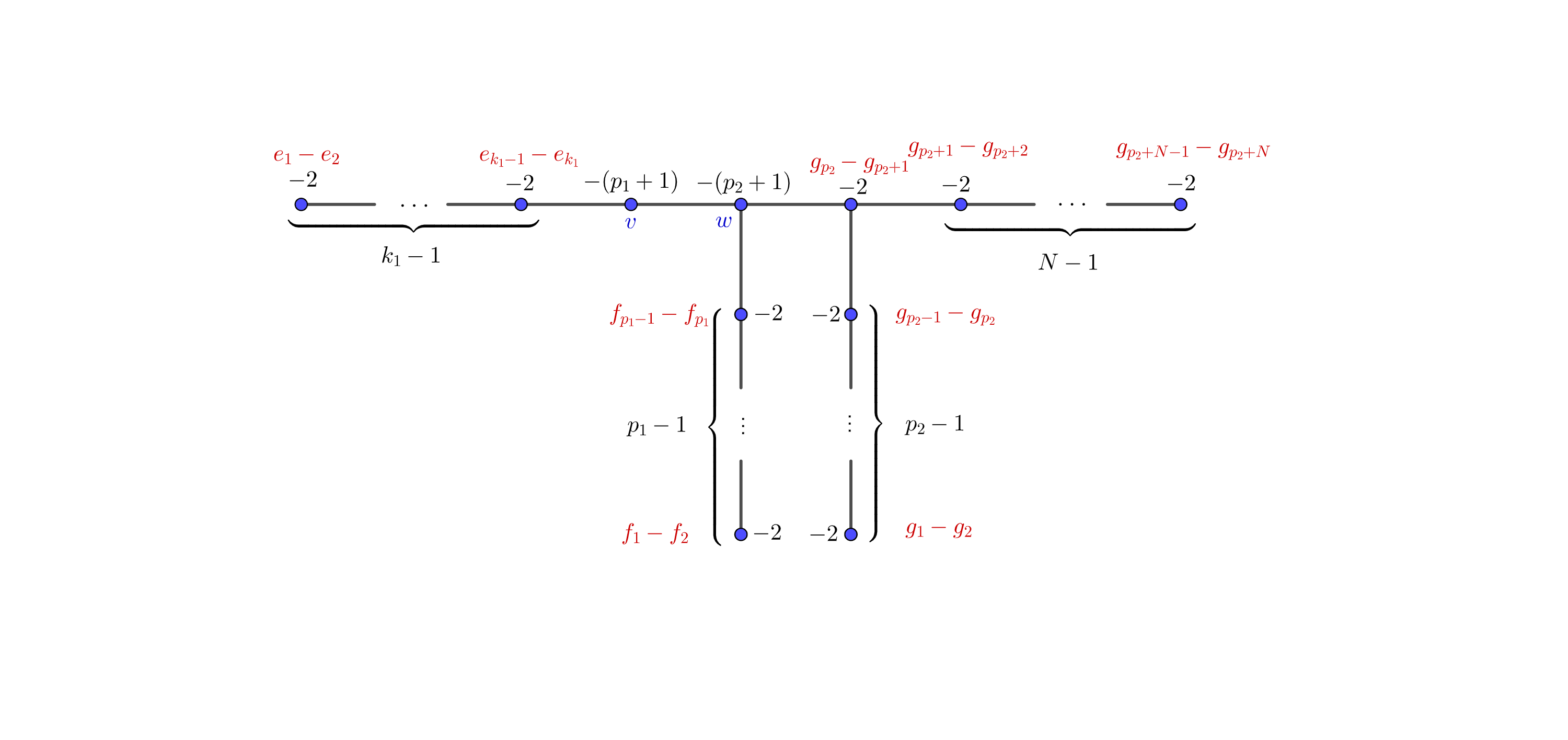}
\caption{Embedding of $S_n^3(T(p_1,k_1p_1+1;p_2,p_1\alpha_1p_2+1)$}
\label{algebraic2}
\end{figure}

\begin{proof}
We start by the case where $\alpha_2 \equiv -1 \pmod{p_2}$, in which case the plumbing graph of $S_n^3(T(p_1,\alpha_1;p_2,\alpha_2))$ looks as in Figure \ref{algebraic1}. The proof is similar to the one of Proposition \ref{propminus1}, with the main difference being that here $v$ and $w$ are adjacent. As before, there are three options for $w$, namely \begin{enumerate}
    \item $w=-g_1+f_{p_1}+h$, and thus $k_1=1$,
    \item $w=g_2+g_3+f_{p_1}$, and thus $p_2=2=N$, or
    \item $w=-f_1-f_2-g_1$ in which case $p_1=3$.
\end{enumerate}
We quickly go through these cases one by one: \begin{enumerate}
    \item If $v$ is hit by $f_{p_1}$, then $v=\pm (f_1+\cdots +f_{p_1})\pm h$, by orthogonality to Leg 1 and the only way to fill out a space of 1. However, this cannot have intersection 1 with $w$. Thus $v=-h+\lambda(g_1+\cdots+g_{p_2+N-1}+h)$. Now, either $u=-(g_1+\cdots+g_{p_2-1})- h$, which cannot be orthogonal to $v$ since $0=\langle v, u \rangle = \lambda p_2-1$ has no solutions, or $u=g_{p_2}+\cdots+g_{p_2+N-1}+\kappa h + \mu (f_1+\cdots +f_{p_1})$. By orthogonality to $w$, we get $\kappa=-\mu$. Orthogonality to $v$ gives $\langle v, u \rangle = (\lambda-1)\mu-\lambda N = 0$. Now, it remains to solve the system \begin{numcases}{}
    p_1+1=-\langle v, v \rangle = \lambda^2(p_2+N-1)+(\lambda-1)^2 \label{eq7}\\
    p_2=-\langle u, u \rangle = N+\mu^2+\mu^2p_1 \label{eq8}\\
    0=\langle v, u \rangle = (\lambda-1)\mu-\lambda N. \label{eq9}
\end{numcases} Note that $\lambda=0$ implies $p_2=-2$, which is not admissible. Thus, by Equation \eqref{eq7}, $p_1\geq p_2$. If $\mu \neq 0$, then by Equation \eqref{eq8}, $p_2>p_1$, which gives a contradiction. If $\mu=0$, then by Equation \eqref{eq9}, $\lambda=0$ or $N=0$, both of which are impossible. Thus we have no solutions in this case.
\item If $w=g_2+g_3+f_{p_1}$, then we have no options for $u$, which here satisfies $-2=\langle u, u \rangle$. Either $u$ is hit by $g_2$ and thus $u=g_2+g_3$, which is not orthogonal to $w$, or $u=-g_1+h$. Then $k_1=1$. Now either $v$ is hit by the $f$'s, implying $v=\pm(f_1+\cdots +f_{p_1})\pm h$, which is not orthogonal to $u$, or $v=\lambda(h+g_1+g_2+g_{3})$, which cannot have intersection 1 with $w$.
\item If $w=-f_1-f_2-g_1$, we have $p_1=3$ and $-4=\langle v, v \rangle$. If $v$ is hit by $f$'s, we have $v=\pm(f_1+f_2+f_3)\pm h$ or $v=\pm(f_1+f_2+f_3)+e_{k_1}$, none of which can have intersection $1$ with $w$. Thus $(g_1+\cdots + g_{p_2+N-1})$ is included in $v$ and since $p_2+N-1\geq 3$ we have four options: \begin{enumerate}
    \item We could have $k_1\geq 2$, $p_2=N=2$ and $v=g_1+g_2+g_3+e_k$. Then either $u=g_2+g_3$ which is not orthogonal to $v$, or $u$ includes $-g_1$, but we cannot fill in the gap of 1.
    \item A similar case is $k_1=1$, $p_2=N=2$ and $v=g_1+g_2+g_3+h$. Then $u=g_2+g_3$ is still impossible, and $u=h-g_1$ is not orthogonal to $w$.
    \item Alternatively, we could have $p_2=2$, $N=3$ and $v=g_1+\cdots + g_4$. Then either $u$ is hit by $g_2$ and includes $g_2+g_3+g_4$ which has lower self-intersection than $-2$, or $u=-g_1+h$, $k_1=1$ and $u$ is not orthogonal to $v$.
    \item Lastly, we could have $p_2=3$, $N=2$ and $v=g_1+g_2+g_3+g_4$. Then either $u=g_3+g_4+h$, which is not orthogonal to $v$, or $u=-g_1-g_2+h$, which is still not orthogonal to $v$.
\end{enumerate}
\end{enumerate}
Thus there are no embeddings if $\alpha_2 \equiv -1 \pmod{p_2}$.

If $\alpha_2 \equiv 1 \pmod{p_2}$, $S_n^3(T(p_1,\alpha_1;p_2, \alpha_2))$ bounds the plumbing graph in Figure \ref{algebraic2}. Once again, the main difference with Proposition \ref{prop1} is that $v$ and $w$ are adjacent. We assume that $p_2\geq 3$ since the other case fits into the case of $\alpha_2 \equiv -1 \pmod{p_2}$. Proposition \ref{minustwochains} forces the partial embedding in red. If $k_1>1$, we have used too many vectors already, so Proposition \ref{donaldson} obstructs the existence of a rational homology 4-ball bounding $S_n^3(T(p_1,\alpha_1;p_2, \alpha_2))$ in this case. We assume $k_1=1$. We have no basis vectors available apart from the $f$'s and $g$'s.

No $f$ can hit $v$ since then all of them would, which would leave us with a gap of 1 that we cannot fill in. Thus $v=\lambda(g_1+\cdots+g_{p_2+N})$. If $w$ is hit by $g_{p_2}$, then $w=f_{p_1}-(g_1+\cdots+g_{p_2})$, which cannot have intersection $1$ with $v$. Thus $w=g_{p_2+1}+\cdots+g_{p_2+N}+\kappa(f_1+\cdots+f_{p_1})$, whose intersection with $v$ is $-\lambda N$, which cannot equal 1. Hence, there are no embeddings when  $\alpha_2 \equiv 1 \pmod{p_2}$ either.
\end{proof}

\bibliographystyle{acm}
\bibliography{lowdim} 

\begin{thebibliography}{10}

\bibitem{aceto20}
{\sc Aceto, P.}
\newblock Rational homology cobordisms of plumbed manifolds.
\newblock {\em Algebr. Geom. Topol. 20}, 3 (2020), 1073--1126.

\bibitem{acetogolla}
{\sc Aceto, P., and Golla, M.}
\newblock Dehn surgeries and rational homology balls.
\newblock {\em Algebr. Geom. Topol. 17}, 1 (2017), 487--527.

\bibitem{GALL}
{\sc {Aceto}, P., {Golla}, M., {Larson}, K., and {Lecuona}, A.~G.}
\newblock {Surgeries on torus knots, rational balls, and cabling}.
\newblock {\em arXiv e-prints\/} (Aug. 2020), arXiv:2008.06760.

\bibitem{bodnar}
{\sc Bodn\'{a}r, J.}
\newblock Classification of rational unicuspidal curves with two {N}ewton
  pairs.
\newblock {\em Acta Math. Hungar. 148}, 2 (2016), 294--299.

\bibitem{casson-harer}
{\sc Casson, A.~J., and Harer, J.~L.}
\newblock Some homology lens spaces which bound rational homology balls.
\newblock {\em Pacific J. Math. 96}, 1 (1981), 23--36.

\bibitem{donaldsonthm}
{\sc Donaldson, S.~K.}
\newblock An application of gauge theory to four dimensional topology.
\newblock {\em J. Differential Geom. 18}, 2 (1983), 279--315.

\bibitem{eisenbudneumann}
{\sc Eisenbud, D., and Neumann, W.}
\newblock {\em Three-dimensional link theory and invariants of plane curve
  singularities}, vol.~110 of {\em Annals of Mathematics Studies}.
\newblock Princeton University Press, Princeton, NJ, 1985.

\bibitem{cameron}
{\sc Gordon, C.~M.}
\newblock Dehn surgery and satellite knots.
\newblock {\em Trans. Amer. Math. Soc. 275}, 2 (1983), 687--708.

\bibitem{kirbylist}
{\sc Kirby, R.}
\newblock Problems in low-dimensional topology.
\newblock In {\em Geometric topology\/} (1997), AMS/IP Stud. Adv.Math., vol. 2,
  Amer. Math. Soc., Providence, R.I., pp.~35--473.

\bibitem{lecuonamontesinos}
{\sc Lecuona, A.~G.}
\newblock On the slice-ribbon conjecture for {M}ontesinos knots.
\newblock {\em Trans. Amer. Math. Soc. 364}, 1 (2012), 233--285.

\bibitem{lecuonalisca11}
{\sc Lecuona, A.~G., and Lisca, P.}
\newblock Stein fillable {S}eifert fibered 3-manifolds.
\newblock {\em Algebr. Geom. Topol. 11}, 2 (2011), 625--642.

\bibitem{liscasingle07}
{\sc Lisca, P.}
\newblock Lens spaces, rational balls and the ribbon conjecture.
\newblock {\em Geom. Topol. 11\/} (2007), 429--472.

\bibitem{liscamultiple07}
{\sc Lisca, P.}
\newblock Sums of lens spaces bounding rational balls.
\newblock {\em Algebr. Geom. Topol. 7\/} (2007), 2141--2164.

\bibitem{neumannplumbingcalculus}
{\sc Neumann, W.~D.}
\newblock A calculus for plumbing applied to the topology of complex surface
  singularities and degenerating complex curves.
\newblock {\em Trans. Amer. Math. Soc. 268}, 2 (1981), 299--344.

\bibitem{neumann89}
{\sc Neumann, W.~D.}
\newblock On bilinear forms represented by trees.
\newblock {\em Bull. Austral. Math. Soc. 40}, 2 (1989), 303--321.

\bibitem{owensstrle12}
{\sc Owens, B., and Strle, S.}
\newblock Dehn surgeries and negative-definite four-manifolds.
\newblock {\em Selecta Math. (N.S.) 18}, 4 (2012), 839--854.

\bibitem{riemenschneider}
{\sc Riemenschneider, O.}
\newblock Deformationen von {Q}uotientensingularit\"{a}ten (nach zyklischen
  {G}ruppen).
\newblock {\em Math. Ann. 209\/} (1974), 211--248.

\bibitem{simone2020classification}
{\sc Simone, J.}
\newblock Classification of torus bundles that bound rational homology circles.
\newblock {\em arXiv preprint arXiv:2006.14986\/} (2020).

\bibitem{simone2021using}
{\sc Simone, J.}
\newblock Using rational homology circles to construct rational homology balls.
\newblock {\em Topology and its Applications 291\/} (2021), 107626.

\end{thebibliography}

\end{document}